\documentclass[dvips]{amsart}
\usepackage[numbers,sort]{natbib}
\usepackage{amsaddr}
\title[Quiver-theoretical approach to dynamical YB maps]
{Quiver-theoretical approach \\ to dynamical Yang-Baxter maps}

\author[D.~K.~Matsumoto]{Diogo Kendy Matsumoto}
\email{diogo-sw@shibaura-it.ac.jp}
\address{Center for Promotion of Educational Innovation \\
  Shibaura Institute of Technology \\
  307 Fukasaku, Minuma-ku, Saitama-shi, Saitama 337-8570, Japan.}

\author[K.~Shimizu]{Kenichi Shimizu}
\email{kshimizu@shibaura-it.ac.jp}
\address{Department of Mathematical Sciences \\
  Shibaura Institute of Technology \\
  307 Fukasaku, Minuma-ku, Saitama-shi, Saitama 337-8570, Japan.}
\thanks{The second author (K.S.) is supported by JSPS KAKENHI Grant Number JP16K17568}

\usepackage{amsmath,amssymb,amscd}
\usepackage{bbm}

\usepackage[all]{xy}

\usepackage{amsthm}
\numberwithin{equation}{section}
\newtheorem{cnt}{}[section]
\theoremstyle{definition}
\newtheorem{definition}         [cnt]{Definition}

\theoremstyle{plain}
\newtheorem{lemma}              [cnt]{Lemma}

\newtheorem{theorem}            [cnt]{Theorem}

\newtheorem*{theorem*}          {Theorem}
\theoremstyle{remark}
\newtheorem{remark}             [cnt]{Remark}
\newtheorem{example}            [cnt]{Example}

\theoremstyle{definition}
\newtheorem*{note*}{Note}

\newcommand{\id}{\mathrm{id}}
\newcommand{\unitobj}{\mathbf{1}}

\newcommand{\rev}{\mathrm{rev}}
\newcommand{\Map}{\mathrm{Map}}
\newcommand{\triactl}{\mathop{\triangleright}}
\newcommand{\triactr}{\mathop{\triangleleft}}
\newcommand{\src}{\mathfrak{s}}
\newcommand{\tgt}{\mathfrak{t}}
\newcommand{\DSet}[1]{\mathsf{DSet}_{#1}}
\newcommand{\bimod}[1]{{#1}\mbox{\textsf{-Bim}}}
\newcommand{\dyact}{\triangleleft}
\newcommand{\dybml}[1]{\underset{#1}{\rightharpoonup}}
\newcommand{\dybmr}[1]{\underset{#1}{\leftharpoonup}}
\newcommand{\Quiv}[1]{\mathsf{Quiv}_{#1}}
\newcommand{\Br}{\mathsf{Br}}
\begin{document}

\begin{abstract}
  A dynamical Yang-Baxter map, introduced by Shibukawa, is a solution of the set-theoretical analogue of the dynamical Yang-Baxter equation. In this paper, we initiate a quiver-theoretical approach for the study of dynamical Yang-Baxter maps. Our key observation is that the category of dynamical sets over a set $\Lambda$, introduced by Shibukawa to establish a categorical framework to deal with dynamical Yang-Baxter maps, can be embedded into the category of quivers with vertices $\Lambda$. By using this embedding, we shed light on Shibukawa's classification result of a certain class of dynamical Yang-Baxter maps and extend his construction to obtain a new class of dynamical Yang-Baxter maps. We also discuss a relation between Shibukawa's bialgebroid associated to a dynamical Yang-Baxter map and Hayashi's weak bialgebra associated to a star-triangular face model.
\end{abstract}

\maketitle

\section{Introduction}

The Yang-Baxter equation was first considered independently by McGuire \cite{MR0161667} and Yang \cite{MR0261870} in their study of one-dimensional many body problems. Finding a (constant) solution of the equation is equivalent to solving the equation
\begin{equation}
  \label{eq:YBE}
  (\sigma \otimes \id_V) (\id_V \otimes \sigma) (\id_V \otimes \sigma)
  = (\id_V \otimes \sigma) (\sigma \otimes \id_V) (\id_V \otimes \sigma)
\end{equation}
for a linear operator $\sigma: V \otimes V \to V \otimes V$ on the two-fold tensor product of a vector space $V$. Thus, abusing terminology, we refer to \eqref{eq:YBE} as the Yang-Baxter equation in this paper. As a natural analogue of \eqref{eq:YBE}, Drinfeld \cite{MR1183474} proposed to investigate the set-theoretical Yang-Baxter equation
\begin{equation}
  \label{eq:set-th-YBE}
  (\sigma \times \id_X) (\id_X \times \sigma) (\id_X \times \sigma)
  = (\id_X \times \sigma) (\sigma \times \id_X) (\id_X \times \sigma)  
\end{equation}
for a map $\sigma: X \times X \to X \times X$ on the two-fold Cartesian product of a set $X$. Solutions of \eqref{eq:set-th-YBE}, called {\em Yang-Baxter maps}, have been studied as well as the solutions of the ordinary Yang-Baxter equation.

Gervais and Neveu \cite{MR747093} have introduced a generalization of the Yang-Baxter equation~\eqref{eq:set-th-YBE}, called the {\em dynamical Yang-Baxter equation}. Let $\mathfrak{h}$ be a commutative Lie algebra, and let $V$ be a diagonalizable $\mathfrak{h}$-module. The dynamical Yang-Baxter equation on $V$ is (equivalent to) the equation
\begin{equation}
  \label{eq:DYBE}
  \sigma(\lambda)_{12} \circ \sigma(\lambda - h^{(1)})_{23} \circ \sigma(\lambda)_{12}
  = \sigma(\lambda - h^{(1)})_{23} \circ \sigma(\lambda)_{12} \circ \sigma(\lambda - h^{(1)})_{23}
\end{equation}
for a family $\{ \sigma(\lambda) \in \mathrm{End}_{\mathfrak{h}}(V \otimes V) \}_{\lambda \in \mathfrak{h}^*}$ of $\mathfrak{h}$-equivariant linear operators. Here, $\sigma(\lambda)_{12}$ and $\sigma(\lambda - h^{(1)})_{23}$ are linear operators on $V \otimes V \otimes V$ defined by
\begin{align*}
  \sigma(\lambda)_{12}(v_1 \otimes v_2 \otimes v_3)
  & = \sigma(\lambda)(v_1 \otimes v_2) \otimes v_3, \\
  \sigma(\lambda - h^{(1)})_{23}(v_1 \otimes v_2 \otimes v_3)
  & = v_1 \otimes \sigma(\lambda - \mathrm{wt}(v_1))(v_2 \otimes v_3)
\end{align*}
for $\lambda \in \mathfrak{h}$ and weight vectors $v_1, v_2, v_3 \in V$. Felder \cite{MR1404026} studied mathematical aspects of this equation. Motivated by its relation to conformal field theory and statistical mechanics, the dynamical Yang-Baxter equation also have been studied extensively; see, {\it e.g.}, the lecture book of Etingof-Latour \cite{MR2142557} and references therein.

It is interesting to study a set-theoretical analogue --- like \eqref{eq:set-th-YBE} --- of the dynamical Yang-Baxter equation~\eqref{eq:DYBE}. Shibukawa gave a mathematical formulation of such an equation: Let $\Lambda$ be a non-empty set, and let $X$ be a set equipped with a map $\Lambda \times X \to \Lambda$ expressed as $(\lambda, x) \mapsto \lambda \dyact x$. The {\em set-theoretical dynamical Yang-Baxter equation} \cite{MR2181454,MR2389797,MR2742743} is the equation
\begin{equation}
  \label{eq:set-th-DYBE}
  \sigma(\lambda)_{12} \circ \sigma(\lambda \dyact X^{(1)})_{23} \circ \sigma(\lambda)_{12}
  = \sigma(\lambda \dyact X^{(1)})_{23} \circ \sigma(\lambda)_{12} \circ \sigma(\lambda \dyact X^{(1)})_{23}
\end{equation}
for a family $\{ \sigma(\lambda): X \times X \to X \times X \}_{\lambda \in \Lambda}$ of maps parametrized by $\Lambda$. Here, $\sigma(\lambda)_{12}$ and $\sigma(\lambda \dyact X^{(1)})_{23}$ are maps on $X \times X \times X$ defined by
\begin{align*}
  \sigma(\lambda)_{12}(x, y, z)
  & = (\sigma(\lambda)(x, y), z), \\
  \sigma(\lambda \dyact X^{(1)})_{23}(x, y, z)
  & = (x, \sigma(\lambda \dyact x)(y, z))
\end{align*}
for $\lambda \in \Lambda$ and $x, y, z \in X$. We usually require $\sigma(\lambda)$ to satisfy the condition
\begin{equation}
  \label{eq:set-th-DYBE-inv-cond}
  (\lambda \dyact (x \dybml{\lambda} y)) \dyact (x \dybmr{\lambda} y)
  = (\lambda \dyact x) \dyact y
  \quad \text{for all $x, y \in X$ and $\lambda \in \Lambda$},
\end{equation}
where the symbols $x \dybml{\lambda} y$ and $x \dybmr{\lambda} y$ are defined by
\begin{equation}
  \label{eq:set-th-DYBE-notation}
  \sigma(\lambda)(x, y) = (x \dybml{\lambda} y, x \dybmr{\lambda} y)
\end{equation}
for $\lambda \in \Lambda$ and $x, y \in X$. The condition~\eqref{eq:set-th-DYBE-inv-cond} corresponds to the $\mathfrak{h}$-equivariance property of $\sigma(\lambda)$ in \eqref{eq:DYBE} and thus it is called the {\em invariance condition} \cite{MR2389797} or the {\em weight zero condition} \cite{MR3528563}.

\begin{definition}[Shibukawa \cite{MR2181454,MR2389797}]
  \label{def:DYB-map}
  A solution $\{ \sigma(\lambda): X \times X \to X \times X \}_{\lambda \in \Lambda}$ of the set-theoretical dynamical Yang-Baxter equation \eqref{eq:set-th-DYBE} satisfying the invariance condition \eqref{eq:set-th-DYBE-inv-cond} is called a {\em dynamical Yang-Baxter map} on $X$.
\end{definition}

Unlike the original dynamical Yang-Baxter equation, it is not known whether dynamical Yang-Baxter maps relate to physical models. At the moment, dynamical Yang-Baxter maps are studied purely from the viewpoint of algebra and combinatorics \cite{MR3448180,MR3374623,MR2846728,MR2742743,MR2588133,MR2389797,MR2181454,MR2969244,MR3528563}.

In this paper, we give a new and systematic method to study dynamical Yang-Baxter maps. Shibukawa \cite{MR2742743} has introduced the monoidal category $\DSet{\Lambda}$ as a useful framework to deal with dynamical Yang-Baxter maps. As Shibukawa pointed out in \cite{MR2742743}, a dynamical Yang-Baxter map is just a braided object of $\DSet{\Lambda}$. Our key observation is that the category $\DSet{\Lambda}$ can be fully embedded into the monoidal category $\Quiv{\Lambda}$ of quivers with vertices $\Lambda$ (Theorem~\ref{thm:dyn-set-embedding}). Thus, through the embedding, a dynamical Yang-Baxter map gives rise to a braided quiver studied by Andruskiewitsch \cite{MR2183847}. Following this scheme, one can utilize quiver-theoretical methods to the study of dynamical Yang-Baxter maps.

\subsection*{Organization of this paper}

We describe the organization of this paper. In Section~\ref{sec:dyn-sets-quiv}, we first introduce the monoidal categories $\DSet{\Lambda}$ and $\Quiv{\Lambda}$ for a non-empty set $\Lambda$. We then construct a fully faithful strong monoidal functor $\mathsf{Q}: \DSet{\Lambda} \to \Quiv{\Lambda}$ (Theorem~\ref{thm:dyn-set-embedding}). This functor turns a dynamical Yang-Baxter map into a braided quiver, {\it i.e.}, a braided object of $\Quiv{\Lambda}$ (Theorem~\ref{thm:BrQ}).

In Section~\ref{sec:sol-lqg}, we reexamine Shibukawa's classification result (cited as Theorem~\ref{thm:dyn-set-PH-DYBE} of this paper) on the dynamical Yang-Baxter maps on a dynamical set of a certain form, which we call a dynamical set of {\em PH type} (Definition~\ref{def:PH-type}). We give a different proof of the classification result from the viewpoint of our quiver-theoretical approach. We also classify the dynamical Yang-Baxter maps on a dynamical set of PH type up to equivalence (Theorem~\ref{thm:DYB-map-PH-equiv}).

In Section~\ref{sec:new-sol}, we consider a class of dynamical sets larger than the class of dynamical sets of PH type. We do not give a classification of the dynamical Yang-Baxter maps on such a dynamical set, but invent a new class of Yang-Baxter maps again by our quiver-theoretical approach.

In Section~\ref{sec:two-weak-bialg}, we concern two constructions of a weak bialgebra from a dynamical Yang-Baxter map $(X, \sigma)$ satisfying a certain condition. On the one hand, Shibukawa \cite{MR3448180} constructed a weak bialgebra $\mathfrak{B}(\sigma)$ from such a dynamical Yang-Baxter map $(X, \sigma)$. On the other hand, we obtain another weak bialgebra $\mathfrak{A}(w_{\sigma})$ by applying Hayashi's construction \cite{MR1623965} to the linearization of the braided quiver associated to $(X, \sigma)$. We give a natural weak bialgebra map $\phi: \mathfrak{A}(w_{\sigma}) \to \mathfrak{B}(\sigma)$ which is not an isomorphism in general (Theorem \ref{thm:two-weak-bialg} and Remark~\ref{rem:two-weak-bialg}).

\subsection*{Acknowledgment}

We are grateful to Takahiro Hayashi and Youichi Shibukawa for helpful discussion. The second author (K.S.) is supported by JSPS KAKENHI Grant Number JP16K17568.

\section{Dynamical sets and quivers}
\label{sec:dyn-sets-quiv}

\subsection{Monoidal categories and functors}

We refer the reader to Mac Lane \cite{MR1712872} for the basic notions in the category theory. A {\em monoidal category} \cite[VII.1]{MR1712872} is a category $\mathcal{C}$ endowed with a functor $\otimes: \mathcal{C} \times \mathcal{C} \to \mathcal{C}$ (called the tensor product), an object $\unitobj \in \mathcal{C}$ (called the unit object), and natural isomorphisms
\begin{equation}
  \label{eq:mon-cat-nat-iso}
  (X \otimes Y) \otimes Z \cong X \otimes (Y \otimes Z)
  \quad \text{and} \quad
  \unitobj \otimes X \cong X \cong X \otimes \unitobj
  \quad (X, Y, Z \in \mathcal{C})
\end{equation}
obeying the pentagon and the triangle axioms. A monoidal category is said to be {\em strict} if the natural isomorphisms~\eqref{eq:mon-cat-nat-iso} are the identities. In view of the Mac Lane coherence theorem, we may assume that all monoidal categories are strict when we discuss the general theory of monoidal categories.

Let $\mathcal{C}$ and $\mathcal{D}$ be (strict) monoidal categories. A {\em monoidal functor} \cite[XI.2]{MR1712872} from $\mathcal{C}$ to $\mathcal{D}$ is a functor $F: \mathcal{C} \to \mathcal{D}$ endowed with a natural transformation
\begin{equation*}
  F^{(2)}_{X,Y}: F(X) \otimes F(Y) \to F(X \otimes Y)
  \quad (X, Y \in \mathcal{C})
\end{equation*}
and a morphism $F^{(0)}: \unitobj \to F(\unitobj)$ satisfying the equations
\begin{gather}
  \label{eq:mon-fun-1}
  F^{(2)}_{X \otimes Y, Z} \circ (F^{(2)}_{X, Y} \otimes \id_{F(Z)})
  = F^{(2)}_{X, Y \otimes Z} \circ (\id_{F(X)} \otimes F^{(2)}_{Y,Z}), \\
  \label{eq:mon-fun-2}
  F^{(2)}_{X, \unitobj} \circ (\id_{F(X)} \otimes F^{(0)})
  = \id_{F(X)}
  = F^{(2)}_{\unitobj, X} \circ (F^{(0)} \otimes \id_{F(X)})
\end{gather}
for all objects $X, Y, Z \in \mathcal{C}$. A monoidal functor $F$ is said to be {\em strong} if $F^{(2)}$ and $F^{(0)}$ are invertible, and {\em strict} if they are the identities.

\subsection{Shibukawa's category of dynamical sets}

Let $\Lambda$ be a non-empty set. To establish a category-theoretical framework dealing with the set-theoretical dynamical Yang-Baxter equation, Shibukawa \cite{MR2742743} introduced the category $\DSet{\Lambda}$. An object of this category is a set $X$ equipped with a map
\begin{equation*}
  \dyact_X: \Lambda \times X \to X, \quad (\lambda, x) \mapsto \lambda \dyact_X x.
\end{equation*}
We will write $\dyact = \dyact_X$ if no confusion arises. Given two objects $X$ and $Y$ of $\DSet{\Lambda}$, a morphism $f: X \to Y$ in $\DSet{\Lambda}$ is a map $f: \Lambda \times X \to Y$ satisfying
\begin{equation}
  \label{eq:dyn-set-morph-def}
  \lambda \dyact f(\lambda, x) = \lambda \dyact x
\end{equation}
for all $\lambda \in \Lambda$, $x \in X$ and $y \in Y$. The composition $g \diamond f$ of morphisms $f: X \to Y$ and $g: Y \to Z$ in $\DSet{\Lambda}$ is defined by
\begin{equation}
  \label{eq:dyn-set-compo}
  (g \diamond f)(\lambda, x) = g(\lambda, f(\lambda, x))
\end{equation}
for $x \in X$ and $\lambda \in \Lambda$. We may regard a morphism $f: X \to Y$ in $\DSet{\Lambda}$ as a family $\{ f(\lambda) \}_{\lambda \in \Lambda}$ of maps from $X$ to $Y$ by $f(\lambda)(x) = f(\lambda, x)$. Then the composition of morphisms in $\DSet{\Lambda}$ is expressed also as
\begin{equation*}
  (g \diamond f)(\lambda) = g(\lambda) \circ f(\lambda) \quad (\lambda \in \Lambda),
\end{equation*}
where $\circ$ is the usual composition of maps. It is easy to see that the second projection $\Lambda \times X \to X$ is the identity morphism on the object $X$. Following Rump's terminology \cite{MR3528563},

\begin{definition}
  We call $\DSet{\Lambda}$ the category of {\em dynamical sets} over $\Lambda$.
\end{definition}

One should be careful when working in $\DSet{\Lambda}$ since a morphism in this category is not a map between the underlying sets. To be sure, we clarify what an isomorphism in this category is:

\begin{lemma}
  \label{lem:dyn-set-isom}
  A morphism $f: X \to Y$ in $\DSet{\Lambda}$ is an isomorphism if and only if the map $f(\lambda): X \to Y$ is bijective for all $\lambda \in \Lambda$.
\end{lemma}
\begin{proof}
  If $f$ is an isomorphism, then there is a morphism $g: Y \to X$ in $\DSet{\Lambda}$ such that $f \diamond g$ and $g \diamond f$ are the identity morphisms. By the definition \eqref{eq:dyn-set-compo} of the composition in $\DSet{\Lambda}$, the map $g(\lambda)$ must be an inverse of $f(\lambda)$. Thus, in particular, the map $f(\lambda)$ is a bijection for all $\lambda \in \Lambda$.

  Suppose, conversely, that the map $f(\lambda): X \to Y$ is a bijection for all $\lambda \in \Lambda$. We define $g(\lambda): Y \to X$ to be the inverse of $f(\lambda)$. Then $g = \{ g(\lambda) \}$ is a morphism from $Y$ to $X$ in $\DSet{\Lambda}$. Indeed, since $f$ is a morphism in $\DSet{\Lambda}$, we have
  \begin{equation*}
    \lambda \dyact g(\lambda, y)
    = \lambda \dyact f(\lambda, g(\lambda, y))
    = \lambda \dyact (f(\lambda)g(\lambda)(y))
    = \lambda \dyact y
  \end{equation*}
  for all $\lambda \in \Lambda$ and $y \in Y$. It is easy to check that $f$ is an isomorphism in $\DSet{\Lambda}$ with inverse $g$.
\end{proof}

We define the object $K_{\Lambda} \in \DSet{\Lambda}$ as follows: As a set, $K_{\Lambda} = \Lambda$. The structure map $\dyact_{K_{\Lambda}}$ for $K_{\Lambda}$ is defined by $\lambda \dyact_{K_{\Lambda}} k = k$ for $\lambda \in \Lambda$ and $k \in K$. This object has the following universal property:

\begin{lemma}
  \label{lem:dyn-set-terminal}
  $K_{\Lambda} \in \DSet{\Lambda}$ is a terminal object.
\end{lemma}
\begin{proof}
  Let $X \in \DSet{\Lambda}$ be an object. There is a morphism
  \begin{equation}
    \label{eq:dyn-set-terminal-map}
    \phi_X: X \to K_{\Lambda},
    \quad \psi_X(\lambda, x) = \lambda \dyact_X x
    \quad (\lambda \in \Lambda, x \in X).
  \end{equation}
  If $f: X \to K_{\Lambda}$ be a morphism in $\DSet{\Lambda}$, then, by~\eqref{eq:dyn-set-morph-def}, we have
  \begin{equation*}
    f(\lambda, x) = \lambda \dyact_{K_{\Lambda}} f(\lambda, x) = \lambda \dyact_X x = \psi_X(\lambda, x)
  \end{equation*}
  for all $\lambda \in \Lambda$ and $x \in X$. Thus the set $\DSet{\Lambda}(X, K_{\Lambda})$ of morphisms from $X$ to $K_{\Lambda}$ consists of one element $\psi_X$ defined in the above.
\end{proof}

Given two objects $X, Y \in \DSet{\Lambda}$, we define their tensor product $X \otimes Y \in \DSet{\Lambda}$ as the Cartesian product $X \times Y$ endowed with the structure map given by
\begin{equation}
  \label{eq:dyn-set-tensor-obj}
  \lambda \dyact_{X \otimes Y} (x, y) = (\lambda \dyact x) \dyact y
\end{equation}
for $\lambda \in \Lambda$, $x \in X$ and $y \in Y$. For morphisms $f: X \to X'$ and $g: Y \to Y'$ in $\DSet{\Lambda}$, their tensor product $f \otimes g: X \otimes Y \to X' \otimes Y'$ is given by
\begin{equation}
  \label{eq:dyn-set-tensor-morph}
  (f \otimes g)(\lambda, (x, y)) = (f(\lambda, x), g(\lambda \dyact x, y))
\end{equation}
for $\lambda \in \Lambda$, $x \in X$ and $y \in Y$. This definition can be expressed also as
\begin{equation}
  \label{eq:dyn-set-tensor-morph-2}
  (f \otimes g)(\lambda) = f(\lambda) \times g(\lambda \dyact X^{(1)})
\end{equation}
if we use the notation similar to that used in the set-theoretical dynamical Yang-Baxter equation \eqref{eq:set-th-DYBE}. The singleton $\unitobj := \{ 1 \}$ is an object of $\DSet{\Lambda}$ with the structure map defined by $\lambda \dyact 1 = \lambda$ for $\lambda \in \Lambda$. There are natural isomorphisms
\begin{gather*}
  a_{X,Y,Z}: (X \otimes Y) \otimes Z \to X \otimes (Y \otimes Z),
  \quad \ell_X: \mathbf{1} \otimes X \to X,
  \quad r_X: X \otimes \mathbf{1} \to X
\end{gather*}
for objects $X, Y, Z \in \DSet{\Lambda}$ defined by
\begin{equation*}
  a_{X,Y,Z}(\lambda)((x, y), z) = (x, (y, z)),
  \quad \ell_{X}(\lambda)(1, x) = x,
  \quad r_{X}(\lambda)(x, x) = x,
\end{equation*}
respectively, for $\lambda \in \Lambda$, $x \in X$, $y \in Y$ and $z \in Z$. The category $\DSet{\Lambda}$ is a monoidal category with the tensor product $\otimes$, the unit object $\unitobj$, and the natural isomorphisms defined in the above.

\begin{remark}
  \label{rem:dyn-set-tensor-2}
  The category $\DSet{\Lambda}$ has an alternative tensor product $\mathbin{\check{\otimes}}$ defined as follows: For objects $X, Y \in \DSet{\Lambda}$, we define $X \mathbin{\check{\otimes}} Y \in \DSet{\Lambda}$ to be the Cartesian product $X \times Y \in \DSet{\Lambda}$ endowed with the structure morphism defined by
  \begin{equation*}
    \lambda \dyact (x, y) = (\lambda \dyact y) \dyact x
    \quad (\lambda \in \Lambda, x \in X, y \in Y).
  \end{equation*}
  For morphisms $f: X \to X'$ and $g: Y \to Y'$ in $\DSet{\Lambda}$, we set
  \begin{equation*}
    (f \mathbin{\check{\otimes}} g)(\lambda, x, y) = f(\lambda \dyact y, x) \times g(\lambda, y)
    \quad (\lambda \in \Lambda, x \in X, y \in Y).
  \end{equation*}
  There is a natural isomorphism
  \begin{equation*}
    \tau_{X,Y}: X \otimes Y \to Y \mathbin{\check{\otimes}} X,
    \quad (x, y) \mapsto (y, x)
    \quad (X, Y \in \DSet{\Lambda}).
  \end{equation*}
  The identity functor on $\DSet{\Lambda}$ gives an isomorphism
  \begin{equation*}
    (\id_{\DSet{\Lambda}}, \tau, \id_{\mathbf{1}}):
    (\DSet{\Lambda}, \otimes^{\rev}, \mathbf{1})
    \xrightarrow{\quad \approx \quad} (\DSet{\Lambda}, \mathbin{\check{\otimes}}, \mathbf{1})
  \end{equation*}
  of monoidal categories, where $X \otimes^{\rev} Y = Y \otimes X$. The two tensor products $\otimes$ and $\mathbin{\check{\otimes}}$ are equivalent in this sense. We use the former one unless otherwise stated.
\end{remark}

\subsection{Category of quivers}
\label{subsec:cat-quivers}

Let $\Lambda$ be a non-empty set. A {\em quiver} over $\Lambda$ is a set $Q$ endowed with two maps $\src_Q, \tgt_Q: Q \to \Lambda$, called the {\em source map} and the {\em target map}, respectively. We write $\src = \src_Q$ and $\tgt = \tgt_Q$ if no confusion arises. An element $a \in Q$ of a quiver $Q$ is called an {\em arrow} from $\src(a)$ to $\tgt(a)$. The diagram
\begin{equation*}
  \xymatrix{
    \overset{\displaystyle \lambda}{\circ}
    \ar[rr]^{\displaystyle a} & & \overset{\displaystyle \mu}{\circ}
  }
\end{equation*}
will be used to mean that $a$ is an arrow from $\lambda$ to $\mu$.

The category $\Quiv{\Lambda}$ of quivers over $\Lambda$ is defined as follows: As its name suggests, an object of this category is a quiver over $\Lambda$. If $Q$ and $Q'$ are objects of $\Quiv{\Lambda}$, then a morphism $f: Q \to Q'$ in $\Quiv{\Lambda}$ is a map $f: Q \to Q'$ such that
\begin{equation*}
  \src(f(a)) = \src(a)
  \quad \text{and} \quad
  \tgt(f(a)) = \tgt(a)
\end{equation*}
for all elements $a \in Q$. The composition of morphisms in $\Quiv{\Lambda}$ is defined by the ordinary composition of maps. The following lemma is easy to prove:

\begin{lemma}
  \label{lem:quiv-iso}
  A morphism in $\Quiv{\Lambda}$ is an isomorphism in $\Quiv{\Lambda}$ if and only if it is a bijection between the underlying sets.
\end{lemma}

Given two objects $Q$ and $R$ of $\Quiv{\Lambda}$, we define their fiber product $Q \times_{\Lambda} R$ by
\begin{equation*}
  Q \times_{\Lambda} R = \{ (a, b) \in Q \times R \mid \tgt(a) = \src(b) \}.
\end{equation*}
The set $Q \times_{\Lambda} R$ is a quiver over $\Lambda$ by
\begin{equation*}
  \src(a, b) = \src(a)
  \quad \text{and} \quad
  \tgt(a, b) = \tgt(b)
  \quad ((a, b) \in Q \times_{\Lambda} R).
\end{equation*}
Given two morphisms $f: Q \to Q'$ and $g: R \to R'$ in $\Quiv{\Lambda}$, we define the map
\begin{equation*}
  f \times_{\Lambda} g: Q \times_{\Lambda} R \to Q' \times_{\Lambda} R'
\end{equation*}
to be the restriction of $f \times g$ to $Q \times_{\Lambda} R$. The fiber product $\times_{\Lambda}$ makes $\Quiv{\Lambda}$ a monoidal category. We note that the unit object of $\Quiv{\Lambda}$ is the set $\Lambda$ with the source and the target map given by $\src = \id_{\Lambda} = \tgt$.

\begin{remark}
  $\Quiv{\Lambda}$ has an alternative tensor product $\mathbin{\check{\times}}_{\Lambda}$ given by
  \begin{equation*}
    Q \mathbin{\check{\times}}_{\Lambda} R = \{ (a, b) \in Q \times R \mid \src(a) = \tgt(b) \}.
  \end{equation*}
  Given $Q \in \Quiv{\Lambda}$, we define the {\em opposite quiver} $\overline{Q}$ of $Q$ to be the quiver obtained by `reversing' all arrows of $Q$. Formally, $\overline{Q}$ is same as $Q$ as a set. The source and the target maps of $\overline{Q}$ are defined by $\src_{\overline{Q}}(\overline{a}) = \tgt_Q(a)$ and $\tgt_{\overline{Q}}(\overline{a}) = \src_Q(a)$ for $a \in Q$, respectively, where an element $a \in Q$ is written as $\overline{a}$ when it is regarded as an element of $\overline{Q}$. The assignment $Q \mapsto \overline{Q}$ gives rise to an isomorphism
  \begin{equation*}
    (\overline{\phantom{m}}, \id, \id_{\Lambda}):
    (\Quiv{\Lambda}, \times_{\Lambda}, \Lambda)
    \xrightarrow{\quad \approx \quad}
    (\Quiv{\Lambda}, \mathbin{\check{\times}}_{\Lambda}, \Lambda)
  \end{equation*}
  of monoidal categories. The two tensor products $\times_{\Lambda}$ and $\mathbin{\check{\times}}_{\Lambda}$ are equivalent in this sense. We always consider the former one in this paper.
\end{remark}

\subsection{Embedding $\DSet{\Lambda}$ into $\Quiv{\Lambda}$}

Let $\Lambda$ be a non-empty set. Given an object $X \in \DSet{\Lambda}$, we define the quiver $\mathsf{Q}(X)$ over $\Lambda$ as follows: As a set,
\begin{equation*}
  \mathsf{Q}(X) = \Lambda \times X.
\end{equation*}
The source and the target map of this quiver are defined by
\begin{equation*}
  \src_{\mathsf{Q}(X)}(\lambda, x) = \lambda
  \quad \text{and} \quad
  \tgt_{\,\mathsf{Q}(X)}(\lambda, x) = \lambda \dyact_X x,
\end{equation*}
respectively, for $\lambda \in \Lambda$ and $x \in X$. Recall that a morphism $f: X \to Y$ in $\DSet{\Lambda}$ is a map $f: \Lambda \times X \to Y$ satisfying~\eqref{eq:dyn-set-morph-def}. For such an $f$, we define
\begin{equation*}
  \mathsf{Q}(f): \mathsf{Q}(X) \to \mathsf{Q}(Y),
  \quad \mathsf{Q}(f)(\lambda, x) = (\lambda, f(\lambda, x))
  \quad (\lambda \in \Lambda, x \in X).
\end{equation*}
We thus obtain a functor $\mathsf{Q}: \DSet{\Lambda} \to \Quiv{\Lambda}$. Our main result in this section states that the functor $\mathsf{Q}$ embeds the monoidal category $\DSet{\Lambda}$ into $\Quiv{\Lambda}$. Namely, we prove:

\begin{theorem}
  \label{thm:dyn-set-embedding}
  The functor $\mathsf{Q}: \DSet{\Lambda} \to \Quiv{\Lambda}$ is a fully faithful strong monoidal functor together with the monoidal structure given by
  \begin{gather}
    \label{eq:Q-mon-str-2}
    \mathsf{Q}^{(2)}_{X,Y}: \mathsf{Q}(X) \times_{\Lambda} \mathsf{Q}(Y) \to \mathsf{Q}(X \otimes Y),
    \quad (\lambda, x, \mu, y) \mapsto (\lambda, x, y), \\
    \label{eq:Q-mon-str-0}
    \mathsf{Q}^{(0)}: \mathsf{Q}(\unitobj) \to \Lambda,
    \quad (\lambda, 1) \mapsto \lambda.
  \end{gather}
  An object $Q \in \Quiv{\Lambda}$ is isomorphic to $\mathsf{Q}(X)$ for some object $X \in \DSet{\Lambda}$ if and only if the following condition is satisfied:
  \begin{equation}
    \label{eq:Q-ess-img-cond}
    \text{The cardinality of the set $\src_Q^{-1}(\lambda)$ does not depend on $\lambda \in \Lambda$}.
  \end{equation}
\end{theorem}
\begin{proof}
  (1) {\it Fully faithfulness}. Let $f: X \to Y$ be a morphism in $\DSet{\Lambda}$. By composing the second projection to $\mathsf{Q}(f)$, we recover the map $f$ itself. Thus $\mathsf{Q}$ is faithful. To show that $\mathsf{Q}$ is full, we let $X, Y \in \DSet{\Lambda}$. Given a morphism $\xi: \mathsf{Q}(X) \to \mathsf{Q}(Y)$ in $\Quiv{\Lambda}$, we define $f: \Lambda \times X \to Y$ to be the composition
  \begin{equation*}
    f = \left( \Lambda \times X \xrightarrow{\quad \xi \quad} \Lambda \times Y \xrightarrow{\quad \mathrm{pr}_2 \quad} Y \right),
  \end{equation*}
  where $\mathrm{pr}_2$ is the second projection. Since $\xi$ preserves the source map, we have
  \begin{equation*}
    \xi(\lambda, x) = (\lambda, f(\lambda, x))
  \end{equation*}
  for all $\lambda \in \Lambda$ and $x \in X$. Since $\xi$ also preserves the target map, we have
  \begin{equation*}
    \lambda \dyact f(\lambda, x)
    = \src(\xi(\lambda, x))
    = \src(\lambda, x)
    = \lambda \dyact x
  \end{equation*}
  for all $\lambda \in \Lambda$ and $x \in X$. Namely, $f$ is a morphism from $X$ to $Y$ in $\DSet{\Lambda}$ such that $\xi = \mathsf{Q}(f)$. Hence the functor $\mathsf{Q}$ is full.

  \medskip \noindent (2) {\it Monoidal structure}.
  We prove that $\mathsf{Q}$ a strong monoidal functor by \eqref{eq:Q-mon-str-2} and~\eqref{eq:Q-mon-str-0}. It is easy to see that $\mathsf{Q}^{(0)}$ is an isomorphism in $\Quiv{\Lambda}$. Let $X$ and $Y$ be objects of $\DSet{\Lambda}$. To investigate the properties of $\mathsf{Q}^{(2)}$, we note that
  \begin{equation*}
    \mathsf{Q}(X) \times_{\Lambda} \mathsf{Q}(Y)
    = \{ (\lambda, x, \mu, y) \in \Lambda \times X \times \Lambda \times Y \mid \mu = \lambda \dyact x \}
  \end{equation*}
  as sets. By this description and Lemma~\ref{lem:quiv-iso}, it is easy to check that $\mathsf{Q}^{(2)}_{X,Y}$ is an isomorphism in $\Quiv{\Lambda}$. For any morphisms $f: X \to X'$ and $g: Y \to Y'$ in $\DSet{\Lambda}$ and an element $(\lambda, x, \mu, y) \in \mathsf{Q}(X) \times_{\Lambda} \mathsf{Q}(Y)$, we have
  \begin{align*}
    \mathsf{Q}_{X',Y'}^{(2)}(\mathsf{Q}(f) \times_{\Lambda} \mathsf{Q}(g))(\lambda,x, \mu,y)
    & = \mathsf{Q}_{X',Y'}^{(2)}(\lambda, f(\lambda, x), \mu, g(\mu, y)) \\
    & = (\lambda, f(\lambda, x), g(\mu, y)) \\
    & = (\lambda, f(\lambda, x), g(\lambda \dyact x, y)) \\
    & = (\lambda, (f \otimes g)(\lambda, x, y)) \\
    & = \mathsf{Q}(f \otimes g) \mathsf{Q}_{X,Y}^{(2)} (\lambda, x, \mu, y).
  \end{align*}
  Thus $\mathsf{Q}^{(2)}_{X,Y}$ is an isomorphism in $\Quiv{\Lambda}$ natural in the variables $X$ and $Y$. It is easy to check that $\mathsf{Q}^{(2)}$ and $\mathsf{Q}^{(0)}$ satisfy the conditions~\eqref{eq:mon-fun-1} and~\eqref{eq:mon-fun-2}.

  \medskip \noindent (3) {\it Determination of the essential image}. Let $\mathcal{I}$ be the essential image of $\mathsf{Q}$, that is, the full subcategory of $\Quiv{\Lambda}$ consisting of all objects $Q \in \Quiv{\Lambda}$ such that $Q \cong \mathsf{Q}(X)$ for some $X \in \DSet{\Lambda}$. It is easy to see that every object of $\mathcal{I}$ satisfies the condition \eqref{eq:Q-ess-img-cond}. We now suppose that an object $Q \in \Quiv{\Lambda}$ satisfies \eqref{eq:Q-ess-img-cond}. We fix $\lambda_0 \in \Lambda$ and set $X = \src_Q^{-1}(\lambda_0)$. For each $\lambda \in \Lambda$, there is a bijection $\phi_{\lambda}: X \to \src_Q^{-1}(\lambda)$ by the assumption. Now we view the set $X$ as an object of $\DSet{\Lambda}$ by defining
  \begin{equation*}
    \lambda \dyact_X x = \tgt_Q(\phi_{\lambda}(x))
    \quad (\lambda \in \Lambda, x \in X).
  \end{equation*}
  Then the map $\mathsf{Q}(X) \to Q$ given by $(\lambda, x) \mapsto \phi_{\lambda}(x)$ is an isomorphism in $\Quiv{\Lambda}$. Thus $Q$ belongs to the essential image of $\mathsf{Q}$.
\end{proof}

\subsection{Dynamical Yang-Baxter maps}
\label{subsec:set-th-dyn-vs-quiv-th}

For a monoidal category $\mathcal{C}$, the category $\Br(\mathcal{C})$ of {\em braided objects} of $\mathcal{C}$ is defined as follows: An object of this category is a pair $(V, c)$ consisting of an object $V \in \mathcal{C}$ and a morphism $c: V \otimes V \to V \otimes V$ in $\mathcal{C}$ satisfying the braid relation
\begin{equation*}
  (c \otimes \id_V) \circ (\id_V \otimes c) \circ (c \otimes \id_V)
  = (\id_V \otimes c) \circ (c \otimes \id_V) \circ (\id_V \otimes c).
\end{equation*}
If $(V, c)$ and $(W, d)$ are objects of $\Br(\mathcal{C})$, then a morphism $f: (V, c) \to (W, d)$ in $\Br(\mathcal{C})$ is a morphism $f: V \to W$ in $\mathcal{C}$ such that
\begin{equation*}
  (f \otimes f) \circ c = d \circ (f \otimes f).
\end{equation*}

\begin{remark}
  In this paper, we do not assume that the morphism-part $c$ of a braided object $(V, c)$ is an isomorphism in the underlying category. Thus our notion of a braided object is what is called a {\em pre-braided object} in \cite{MR3081627,MR3105304,2016arXiv160708081L}. As pointed out in these papers, non-invertible braided objects naturally arise in studying several algebraic structures.
\end{remark}

The equation \eqref{eq:YBE} says that the pair $(V, \sigma)$ is a braided object of the category of vector spaces. Similarly, the set-theoretical Yang-Baxter equation \eqref{eq:set-th-YBE} is equivalent to that the pair $(X, \sigma)$ is a braided object of the category of sets. The dynamical Yang-Baxter equation~\eqref{eq:DYBE} can also be regarded as a braided object in the monoidal category $\mathcal{V}_{\mathfrak{h}}$ introduced in \cite{MR1645196}.

The category $\DSet{\Lambda}$ has a role of $\mathcal{V}_{\mathfrak{h}}$ in the study of the set-theoretical Yang-Baxter equation \eqref{eq:set-th-DYBE}. As we have introduced in Definition~\ref{def:DYB-map}, a {\em dynamical Yang-Baxter map} on a dynamical set $X$ is a family $\sigma = \{ \sigma(\lambda) \}_{\lambda \in \Lambda}$ of maps satisfying \eqref{eq:set-th-DYBE} and \eqref{eq:set-th-DYBE-inv-cond}. Shibukawa \cite{MR2742743} pointed out that $\sigma$ is a dynamical Yang-Baxter map on $X$ if and only if $(X, \sigma)$ is a braided object of $\DSet{\Lambda}$.

Now let $\Lambda$ be a non-empty set. Since the assignment $\mathcal{C} \mapsto \Br(\mathcal{C})$ is functorial with respect to strong monoidal functors, we have:

\begin{theorem}
  \label{thm:BrQ}
  The functor $\mathsf{Q}$ of Theorem~\ref{thm:dyn-set-embedding} induces a fully faithful functor
  \begin{equation*}
    \Br(\mathsf{Q}): \Br(\DSet{\Lambda}) \to \Br(\Quiv{\Lambda}).
  \end{equation*}
\end{theorem}

In more detail, the functor $\Br(\mathsf{Q})$ sends a braided object $(X, \sigma)$ of $\DSet{\Lambda}$ to the braided object $(\mathsf{Q}(X), \widetilde{\sigma})$ of $\Quiv{\Lambda}$, where
\begin{equation}
  \label{eq:def-sigma-tilde-1}
  \widetilde{\sigma} := (\mathsf{Q}^{(2)}_{X,X})^{-1} \circ \mathsf{Q}(\sigma) \circ \mathsf{Q}^{(2)}_{X,X}.
\end{equation}
If we use the notations given by~\eqref{eq:set-th-DYBE-notation}, then the morphism $\widetilde{\sigma}$ is given by
\begin{equation}
  \label{eq:def-sigma-tilde-2}
  \widetilde{\sigma}((\lambda, x), (\mu, y))
  = ((\lambda, x \dybml{\lambda} y), (\lambda \dyact (x \dybml{\lambda} y), x \dybmr{\lambda} y))
\end{equation}
for $\lambda, \mu \in \Lambda$ and $x, y \in X$ with $\mu = \lambda \dyact x$.

A braided object of $\Quiv{\Lambda}$ is called a {\em braided quiver} over $\Lambda$ \cite{MR2183847}. The above theorem says that finding a dynamical Yang-Baxter map on $X \in \DSet{\Lambda}$ is equivalent to finding a braided quiver $(A, \sigma)$ such that $A$ is isomorphic to $\mathsf{Q}(X)$. In practice, the category $\Quiv{\Lambda}$ is easier to deal with than $\DSet{\Lambda}$. 

\begin{remark}
  \label{rem:set-th-DYBE-2}
  Let $X$ be an object of $\DSet{\Lambda}$, and let $\{ \check{\sigma}(\lambda): X \times X \to X \times X \}_{\lambda \in \Lambda}$ be a family of maps. We consider the condition
  \begin{equation}
    \label{eq:set-th-DYBE-2-inv-cond}
    (\lambda \dyact R^{(\lambda)}_{x y})
    \dyact L^{(\lambda)}_{x y} = (\lambda \dyact y) \dyact x
    \quad (\lambda \in \Lambda, x, y \in X),
  \end{equation}
  where $(L^{(\lambda)}_{x y}, R^{(\lambda)}_{x y}) = \check{\sigma}(\lambda)(x, y)$. In \cite{MR3448180}, the equation
  \begin{equation}
    \label{eq:set-th-DYBE-2}
    \check{\sigma}(\lambda)_{23}
    \circ \check{\sigma}(\lambda \dyact X^{(3)})_{12}
    \circ \check{\sigma}(\lambda)_{23}
    = \check{\sigma}(\lambda \dyact X^{(3)})_{12}
    \circ \check{\sigma}(\lambda)_{23}
    \circ \check{\sigma}(\lambda \dyact X^{(3)})_{12}
  \end{equation}
  is called the set-theoretical dynamical Yang-Baxter equation instead of~\eqref{eq:set-th-DYBE}. Recall from Remark~\ref{rem:dyn-set-tensor-2} that $\DSet{\Lambda}$ has the alternative tensor product $\check{\otimes}$. The condition \eqref{eq:set-th-DYBE-2-inv-cond} says that $\check{\sigma}$ is a morphism $\check{\sigma}: X \mathbin{\check{\otimes}} X \to X \mathbin{\check{\otimes}} X$ in $\DSet{\Lambda}$, and the equation \eqref{eq:set-th-DYBE-2} is equivalent to
  \begin{equation*}
    (\id_{X} \mathbin{\check{\otimes}} \check{\sigma})
    (\check{\sigma} \mathbin{\check{\otimes}} \id_{X})
    (\id_{X} \mathbin{\check{\otimes}} \check{\sigma})
    =
    (\check{\sigma} \mathbin{\check{\otimes}} \id_{X})
    (\id_{X} \mathbin{\check{\otimes}} \check{\sigma})
    (\check{\sigma} \mathbin{\check{\otimes}} \id_{X})
  \end{equation*}
  in $(\DSet{\Lambda}, \mathbin{\check{\otimes}}, \mathbf{1})$. Thus $\check{\sigma}$ satisfies~\eqref{eq:set-th-DYBE-2-inv-cond} and \eqref{eq:set-th-DYBE-2} if and only if
  \begin{equation*}
    \sigma(\lambda)(x, y) = (R^{(\lambda)}_{y x}, L^{(\lambda)}_{y x})
    \quad (\lambda \in \Lambda, x, y \in X)
  \end{equation*}
  is a dynamical Yang-Baxter map on $X$.
\end{remark}

\section{Solutions arising from left quasigroups}
\label{sec:sol-lqg}

\newcommand{\blankternary}{\langle -, -, -\rangle}

\subsection{Dynamical sets of PH type}

We fix a non-empty set $\Lambda$ throughout this section. A {\em left quasigroup} is a set $L$ endowed with a binary operation $*_L$ such that the map $L \to L$, $x \mapsto a *_L x$ is bijective for all $a \in L$. For a left quasigroup $L$, the {\em left division} of $L$ is the binary operation $\backslash_L$ on $L$ determined by
\begin{equation*}
  a \backslash_L (a *_L x) = x = a *_L (a \backslash_L b)
  \quad (\forall a, x \in L).
\end{equation*}
Many known examples of dynamical Yang-Baxter maps are constructed by using left quasigroups \cite{MR2181454,MR2389797,MR2742743,MR2846728,MR3374623}. Motivated by these results, we consider the following class of dynamical sets:

\begin{definition}
  \label{def:PH-type}
  We say that a dynamical set $X \in \DSet{\Lambda}$ is of {\em principal homogeneous type}, or {\em PH type} for short, if the following map is bijective:
  \begin{equation}
    \label{eq:PH-map}
    \Lambda \times X \to \Lambda \times \Lambda, \quad (\lambda, x) \mapsto (\lambda, \lambda \dyact x).
  \end{equation}
\end{definition}

If $*$ is a binary operation on $\Lambda$ such that $(\Lambda, *)$ is a left quasigroup, then the set $X = \Lambda$ is a dynamical set over $\Lambda$ of PH type by $\dyact_X = *$. For example, the binary operation $*$ given by $\lambda * \mu = \mu$ makes $\Lambda$ a left quasigroup, and the resulting dynamical set is the terminal object $K_{\Lambda} \in \DSet{\Lambda}$ mentioned in Lemma~\ref{lem:dyn-set-terminal}. Every dynamical set over $\Lambda$ of PH type can be obtained from a left quasigroup structure on $\Lambda$ in this way. In fact, the following stronger statement holds true:

\begin{lemma}
  \label{lem:dyn-set-PH-unique}
  An object $X \in \DSet{\Lambda}$ is of PH type if and only if $X \cong K_{\Lambda}$.
\end{lemma}
\begin{proof}
  There is a morphism $\psi_X: X \to K_{\Lambda}$ in $\DSet{\Lambda}$ given by~\eqref{eq:dyn-set-terminal-map}. The map~\eqref{eq:PH-map} is just the morphism $\mathsf{Q}(\psi_X): \mathsf{Q}(X) \to \mathsf{Q}(K_{\Lambda})$ in $\Quiv{\Lambda}$. By the faithfulness of $\mathsf{Q}$, it is clear that $X$ is of PH type if and only if $\psi_X$ is an isomorphism.
\end{proof}

Thus, strange as it may sound, a dynamical set over $\Lambda$ of PH type is unique up to isomorphism. Nevertheless we have introduced this terminology, since we will deal with many such objects having various different appearances.

\subsection{Classification of the solutions}

Let $X$ be a dynamical set over $\Lambda$ of PH type. As an analogue of the left division of a left quasigroup, for elements $\lambda, \lambda' \in \Lambda$, we define the {\em left division} $\lambda \backslash_X \lambda' \in X$ to be the unique element of $X$ such that
\begin{equation}
  \label{eq:left-div-dyn-set-PH}
  \lambda \dyact (\lambda \backslash_X \lambda') = \lambda'.
\end{equation}
We recall Shibukawa's classification result of the dynamical Yang-Baxter maps on $X$ in a slightly reformulated form. A {\em ternary operation} on a set $M$ is just a map from $M \times M \times M$ to $M$. Suppose that a ternary operation
\begin{equation}
  \label{eq:ternary-on-Lambda}
  \blankternary: \Lambda \times \Lambda \times \Lambda \to \Lambda, \quad (a, b, c) \mapsto \langle a, b, c \rangle
  \quad (a, b, c \in \Lambda)
\end{equation}
on the set $\Lambda$ is given. For $\lambda \in \Lambda$ and $x, y \in X$, we define
\begin{equation}
  \label{eq:sigma-ternary-def-1}
  \sigma(\lambda)(x, y) = (x \dybml{\lambda} y, x \dybmr{\lambda} y),
\end{equation}
where
\begin{align}
  \label{eq:sigma-ternary-def-2}
  x \dybml{\lambda} y
  & = \lambda \Big\backslash_{X} \langle \lambda, \lambda \dyact x, (\lambda \dyact x) \dyact y \rangle, \\
  \label{eq:sigma-ternary-def-3}
  x \dybmr{\lambda} y
  & = \langle \lambda, \lambda \dyact x, (\lambda \dyact x) \dyact y \rangle
    \Big \backslash_X ((\lambda \dyact x) \dyact y).
\end{align}
We remark that the symbols $x \dybml{\lambda} y$ and $x \dybmr{\lambda} y$ are defined so that the family
\begin{equation}
  \label{eq:sigma-ternary-def-4}
  \sigma := \{ \sigma(\lambda): X \times X \to X \times X \}_{\lambda \in \Lambda}
\end{equation}
of maps is a morphism $X \otimes X \to X \otimes X$ in $\DSet{\Lambda}$. Shibukawa \cite[Theorem 3.2]{MR2389797} proved:

\begin{theorem}
  \label{thm:dyn-set-PH-DYBE}
  The morphism $\sigma: X \otimes X \to X \otimes X$ is a dynamical Yang-Baxter map on $X$ if and only if the ternary operation \eqref{eq:ternary-on-Lambda} satisfies the equations
  \begin{gather}
    \label{eq:Shibu-ternary-1}
    \langle a, \langle a, b, c \rangle, \langle \langle a, b, c \rangle, c, d \rangle \rangle
    = \langle a, b, \langle b, c, d \rangle \rangle, \\
    \label{eq:Shibu-ternary-2}
    \langle \langle a, b, c \rangle, c, d \rangle
    = \langle \langle a, b, \langle b, c, d \rangle \rangle, \langle b, c, d \rangle, d \rangle
    \phantom{,}
  \end{gather}
  for all $a, b, c, d \in \Lambda$. Moreover, every dynamical Yang-Baxter map on $X$ can be obtained in this way from such a ternary operation on $\Lambda$.
\end{theorem}

To avoid confusion with the composition of maps, we use the symbol $\diamond$ to express the composition of morphisms in $\DSet{\Lambda}$. We say that two dynamical Yang-Baxter maps $(X, \sigma)$ and $(Y, \tau)$ are {\em equivalent} if they are isomorphic in $\Br(\DSet{\Lambda})$, that is, if there is an isomorphism $f: X \to Y$ in $\DSet{\Lambda}$ such that
\begin{equation*}
  (f \otimes f) \diamond \sigma = \tau \diamond (f \otimes f).
\end{equation*}
As in Lemma~\ref{lem:dyn-set-terminal}, we denote by $\psi_X: X \to K_{\Lambda}$ the unique morphism in $\DSet{\Lambda}$. In this section, we also give the following classification up to equivalence.

\begin{theorem}
  \label{thm:DYB-map-PH-equiv}
  Let $X_i$ \textup{(}$i = 1, 2$\textup{)} be a dynamical set over $\Lambda$ of PH type, and let $T_i$ be a ternary operation on $\Lambda$ satisfying \eqref{eq:Shibu-ternary-1} and~\eqref{eq:Shibu-ternary-2}. We define $\sigma_i$ by \eqref{eq:sigma-ternary-def-1}--\eqref{eq:sigma-ternary-def-4} with $X = X_i$ and $\blankternary = T_i$. Then $(X_1, \sigma_1)$ and $(X_2, \sigma_2)$ are equivalent if and only if $T_1 = T_2$. If this is the case, then we have
  \begin{equation*}
    \sigma_{2} = (\psi_{21} \otimes \psi_{21}) \diamond \sigma_1 \diamond (\psi_{12} \otimes \psi_{12}),
  \end{equation*}
  where $\psi_{i j} = (\psi_{X_i})^{-1} \diamond \psi_{X_j}$ \textup{(}$i, j = 1, 2$\textup{)}.
\end{theorem}

Explicitly, the morphism $\psi_{i j}: X_j \to X_i$ in the above is given by
\begin{equation*}
  \psi_{i j}(\lambda, x) = \lambda \backslash_{X_i} (\lambda \dyact_{X_j} x)
  \quad (\lambda \in \Lambda, x \in X_j).
\end{equation*}

We give proofs of these theorems by emphasizing the relation between dynamical sets over $\Lambda$ and quivers over $\Lambda$. We will see that the quivers corresponding to the dynamical sets of PH type are of very simple form, and the equations \eqref{eq:Shibu-ternary-1} and \eqref{eq:Shibu-ternary-2} naturally arise from the braid equation on such quivers. Our point of view also clarifies relations between the properties of the ternary operation on $\Lambda$ and the properties of the resulting dynamical Yang-Baxter maps.

\subsection{Morphisms on a complete quiver}

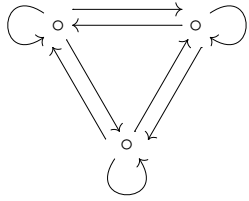
\begin{figure}
  \begin{equation*}
    \SelectTips{cm}{}
    \xymatrix@R=34.64pt@C=15.00pt{
      \circ \ar@<+6pt>[rr] \ar@<+0pt>[rd] \ar@(ul,dl) & &
      \circ \ar@<+0pt>[ll] \ar@<+6pt>[ld] \ar@(ur,dr) \\ &
      \circ \ar@<+0pt>[ru] \ar@<+6pt>[lu] \ar@(dl,dr) \\
    }
  \end{equation*}
  \caption{Complete quiver over a set of three elements}
  \label{fig:complete-quiv}
\end{figure}

A {\em complete quiver} over $\Lambda$ is a quiver $A$ over $\Lambda$ such that there exists a unique arrow of $A$ from $\lambda$ to $\mu$ for each pair $(\lambda, \mu)$ of elements of $\Lambda$. We warn that a complete quiver is {\em not} one whose underlying graph is a complete graph. For example, if $\Lambda$ consists of three elements, a complete quiver over $\Lambda$ is depicted as in Figure~\ref{fig:complete-quiv}.

A complete quiver over $\Lambda$ is unique up to isomorphism. More precisely, if $A$ is a complete quiver over $\Lambda$, then there is an isomorphism
\begin{equation*}
  A \to \mathsf{Q}(K), \quad a \mapsto (\src(a), \tgt(a)) \quad (a \in A)
\end{equation*}
of quivers, where $K = K_{\Lambda}$ is the object of $\DSet{\Lambda}$ discussed in Lemma~\ref{lem:dyn-set-terminal}. In this sense, a complete quiver can be thought of as a quiver-theoretical counterpart of the notion of a dynamical set of PH type.

Now let $A$ be a complete quiver over $\Lambda$. For $\lambda, \mu \in \Lambda$, we denote by $\lambda \to \mu$ the unique element $a \in A$ such that $\src(a) = \lambda$ and $\tgt(a) = \mu$. Then we have
\begin{equation*}
  \underbrace{A \times_{\Lambda} \dotsb \times_{\Lambda} A}_{\text{$m$ times}}
  = \{ (\lambda_1 \to \lambda_2, \lambda_2 \to \lambda_3, \dotsc, \lambda_{m} \to \lambda_{m+1}) \mid \lambda_1, \dotsc, \lambda_{m+1} \in \Lambda \}
\end{equation*}
for a positive integer $m$. For simplicity, we write
\begin{equation}
  \label{eq:uniq-arr-notation}
  \lambda_1 \to \lambda_2 \to \dotsb \to \lambda_m
  := (\lambda_1 \to \lambda_2, \lambda_2 \to \lambda_3, \dotsc, \lambda_{m - 1} \to \lambda_m).
\end{equation}
Let $f: A \times_{\Lambda} A \to A \times_{\Lambda} A$ be a morphism in $\Quiv{\Lambda}$, and let $\lambda, \mu, \nu \in \Lambda$. Since $f$ preserves the source and the target maps, there is a unique element $\mu' \in \Lambda$ such that $f(\lambda \to \mu \to \nu) = \lambda \to \mu' \to \nu$. We thus obtain a ternary operation $\blankternary_f$ on $\Lambda$ uniquely determined by the property that
\begin{equation*}
  f(\lambda \to \mu \to \nu) = a \to \langle \lambda, \mu, \nu \rangle_f \to \nu
\end{equation*}
for all $\lambda, \mu, \nu \in \Lambda$. It is obvious that there holds:

\begin{lemma}
  \label{lem:mor-ternary}
  The assignment $f \mapsto \blankternary_{f}$ gives a bijection
  \begin{equation}
    \label{eq:def-mor-ternary-2}
    \Quiv{\Lambda}(A \times_{\Lambda} A, A \times_{\Lambda} A)
    \xrightarrow{\quad \cong \quad}
    \Map(\Lambda \times \Lambda \times \Lambda, \Lambda).
  \end{equation}
\end{lemma}

\subsection{Quiver-theoretical aspect of the construction}

Let $X$ be a dynamical set over $\Lambda$ of PH type. We explain that the formulas \eqref{eq:sigma-ternary-def-1}--\eqref{eq:sigma-ternary-def-4} naturally arise from our quiver-theoretical approach. A key observation is that $A := \mathsf{Q}(X)$ is a complete quiver over $\Lambda$. Since the functor $\mathsf{Q}$ is fully faithful, we have a bijection
\begin{equation}
  \label{eq:equiv-X2-A2}
  \begin{aligned}
    \DSet{\Lambda}(X \otimes X, X \otimes X)
    & \to \Quiv{\Lambda}(A \times_{\Lambda} A, A \times_{\Lambda} A), \\
    f & \mapsto (\mathsf{Q}^{(2)}_{X,X})^{-1} \circ \mathsf{Q}(f) \circ \mathsf{Q}_{X,X}^{(2)},
  \end{aligned}
\end{equation}
where $\mathsf{Q}^{(2)}$ is the monoidal structure of $\mathsf{Q}$. We now obtain a bijection
\begin{equation}
  \label{eq:equiv-ternary-X2}
  \Map(\Lambda \times \Lambda \times \Lambda, \Lambda)
  \xrightarrow{\quad \cong \quad}
  \DSet{\Lambda}(X \otimes X, X \otimes X)
\end{equation}
by composing the bijections \eqref{eq:def-mor-ternary-2} and~\eqref{eq:equiv-X2-A2}. This bijection clarifies the meaning of the defining formulas~\eqref{eq:sigma-ternary-def-1}--\eqref{eq:sigma-ternary-def-4} of the morphism $\sigma$ of Theorem~\ref{thm:dyn-set-PH-DYBE}.

\begin{lemma}
  \label{lem:equiv-ternary-X2}
  The morphism $\sigma: X \otimes X \to X \otimes X$ in $\DSet{\Lambda}$ corresponding to a ternary operation $\blankternary$ on $\Lambda$ via \eqref{eq:equiv-ternary-X2} is given by the formulas~\eqref{eq:sigma-ternary-def-1}--\eqref{eq:sigma-ternary-def-4}.
\end{lemma}
\begin{proof}
  Let $\blankternary$ be a ternary operation on $\Lambda$, and let $\sigma: X \otimes X \to X \otimes X$ be the morphism in $\DSet{\Lambda}$ corresponding to $\blankternary$ via~\eqref{eq:equiv-ternary-X2}. We write
  \begin{equation*}
    \sigma(\lambda)(x, y) = (L^{(\lambda)}_{x y}, R^{(\lambda)}_{x y})
  \end{equation*}
  for $\lambda \in \Lambda$ and $x, y \in X$. Let $\widetilde{\sigma}: A \times_{\Lambda} A \to A \times_{\Lambda} A$ be the morphism corresponding to $\sigma$ via~\eqref{eq:equiv-X2-A2}. Then we have
  \begin{equation}
    \label{eq:sigma-tilde-formula-1}
    \widetilde{\sigma}((\lambda, x), (\mu, y))
    = ((\lambda, L^{(\lambda)}_{x y}), (\lambda \dyact L^{(\lambda)}_{x y}, R^{(\lambda)}_{x y}))
  \end{equation}
  for $\lambda, \mu \in \Lambda$ and $x, y \in X$ with $\mu = \lambda \dyact x$.

  By definition, $\widetilde{\sigma}$ corresponds to the ternary operation $\blankternary$ via~\eqref{eq:def-mor-ternary-2}. As in the previous subsection, we use the notation~\eqref{eq:uniq-arr-notation} to express elements of the power of $A := \mathsf{Q}(X)$ with respect to $\times_{\Lambda}$. We then have
  \begin{equation}
    \label{eq:uniq-arr-formula}
    \lambda \to (\lambda \dyact x) = (\lambda, x)
    \quad \text{and} \quad
    \lambda \to \lambda' = (\lambda, \lambda \backslash_X \lambda')
    \quad (x \in X, \lambda, \lambda' \in \Lambda),
  \end{equation}
  where $\backslash_X$ is the left division defined by~\eqref{eq:left-div-dyn-set-PH}. By \eqref{eq:Shibu-ternary-1}, \eqref{eq:uniq-arr-notation} and~\eqref{eq:uniq-arr-formula}, we also have the following expression of $\widetilde{\sigma}$\,:
  \begin{equation}
    \label{eq:sigma-tilde-formula-2}
    \begin{aligned}
      \widetilde{\sigma}(a \to b \to c)
      & = a \to \langle a, b, c \rangle \to c \\
      & = (a \to \langle a, b, c \rangle,  \langle a, b, c \rangle \to c) \\
      & = ((a, a \backslash_X \langle a, b, c \rangle), (\langle a, b, c \rangle, \langle a, b, c \rangle \backslash_X c)).
    \end{aligned}
  \end{equation}
  Now we consider the case where $a = \lambda$, $b = \lambda \dyact x$ and $c = (\lambda \dyact x) \dyact y$ for some $\lambda \in \Lambda$ and $x, y \in X$. Then, by~\eqref{eq:uniq-arr-formula}, we have
  \begin{equation*}
    a \to b \to c
    = ((\lambda, x), (\lambda \dyact x, y))
  \end{equation*}
  as elements of $A \times_{\Lambda} A$. Comparing~\eqref{eq:sigma-tilde-formula-1} with~\eqref{eq:sigma-tilde-formula-2}, we conclude that
  \begin{align*}
    L^{(\lambda)}_{x y}
    & = a \backslash_X \langle a, b, c \rangle
      = \lambda \backslash_X \langle \lambda, \lambda \dyact x, (\lambda \dyact x) \dyact y \rangle \\
    R^{(\lambda)}_{x y}
    & = \langle a, b, c \rangle \backslash_X c
      = \langle \lambda, \lambda \dyact x, (\lambda \dyact x) \dyact y \rangle \backslash_X ((\lambda \dyact x) \dyact y).
  \end{align*}
  Thus $\sigma$ is given by the formulas~\eqref{eq:sigma-ternary-def-1}--\eqref{eq:sigma-ternary-def-4}.
\end{proof}

\subsection{Proof of Theorem~\ref{thm:dyn-set-PH-DYBE}}
\label{subsec:proof-thm-33}

Let, as in Theorem~\ref{thm:dyn-set-PH-DYBE}, $X$ be a dynamical set over $\Lambda$ of PH type. Now we give a proof of Shibukawa's classification result on the dynamical Yang-Baxter maps on $X$.

We denote by $A = \mathsf{Q}(X)$ the corresponding quiver. Let $\sigma: X \otimes X \to X \otimes X$ be a morphism in $\DSet{\Lambda}$, and let $\widetilde{\sigma}: A \times_{\Lambda} A \to A \times_{\Lambda} A$ be the morphism corresponding to $\sigma$ via~\eqref{eq:equiv-X2-A2}. The argument of the previous subsection shows that there is a unique ternary operation $\blankternary$ on $\Lambda$ such that the morphism $\sigma$ is expressed by \eqref{eq:sigma-ternary-def-1}--\eqref{eq:sigma-ternary-def-4}. The morphism $\widetilde{\sigma}$ is also expressed in terms of the same ternary operation on $\Lambda$. The point is that, compared to $\sigma$, the morphism $\widetilde{\sigma}$ has fairly easier expression: It can be computed by the following rule:
\begin{equation}
  \label{eq:attaching-rule}
  \widetilde{\sigma}(a \to b \to c) = a \to \langle a, b, c \rangle \to c
  \quad (a, b, c \in \Lambda).
\end{equation}

\begin{proof}[Proof of Theorem~\ref{thm:dyn-set-PH-DYBE}]
  Since $\mathsf{Q}$ is a strong monoidal functor, the morphism $\sigma$ is a dynamical Yang-Baxter map if and only if the morphism $\widetilde{\sigma}$ satisfies the braid relation. Now we set $\widetilde{\sigma}_{12} = \widetilde{\sigma} \times_{\Lambda} \id_{A}$ and $\widetilde{\sigma}_{23} = \id_A \times_{\Lambda} \widetilde{\sigma}$. By the rule~\eqref{eq:attaching-rule}, we compute, without any difficulty,
  \begin{align*}
    & \widetilde{\sigma}_{12} \widetilde{\sigma}_{23} \widetilde{\sigma}_{12} (a \to b \to c \to d) \\
    & \quad = \widetilde{\sigma}_{12} \widetilde{\sigma}_{23} (a \to \langle a, b, c \rangle \to c \to d) \\
    & \quad = \widetilde{\sigma}_{12} (a \to \langle a, b, c \rangle \to \langle \langle a, b, c \rangle, c, d \rangle \to d) \\
    & \quad = a \to \langle a, \langle a, b, c \rangle, \langle \langle a, b, c \rangle, c, d \rangle \rangle
      \to \langle \langle a, b, c \rangle, c, d \rangle \to d, \\
    & \widetilde{\sigma}_{23} \widetilde{\sigma}_{12} \widetilde{\sigma}_{23} (a \to b \to c \to d) \\
    & \quad = \widetilde{\sigma}_{23} \widetilde{\sigma}_{12} (a \to b \to \langle b, c, d \rangle \to d) \\
    & \quad = \widetilde{\sigma}_{23} (a \to \langle a, b, \langle b, c, d \rangle \rangle \to \langle b, c, d \rangle \to d) \\
    & \quad = a \to \langle a, b, \langle b, c, d \rangle \rangle
      \to \langle \langle a, b, \langle b, c, d \rangle \rangle, \langle b, c, d \rangle, d \rangle \to d
  \end{align*}
  for $a, b, c, d \in \Lambda$. Thus $\sigma$ is a dynamical Yang-Baxter map if and only if \eqref{eq:Shibu-ternary-1} and~\eqref{eq:Shibu-ternary-2} are satisfied for all $a, b, c, d \in \Lambda$. By Lemma~\ref{lem:equiv-ternary-X2}, every dynamical Yang-Baxter map on $X$ is obtained in this way from a ternary operation on $\Lambda$.
\end{proof}

\subsection{Proof of Theorem~\ref{thm:DYB-map-PH-equiv}}
\label{subsec:proof-thm-34}

We give a proof of Theorem~\ref{thm:DYB-map-PH-equiv}. Let $K = K_{\Lambda}$ be the terminal object considered in Lemma~\ref{lem:dyn-set-terminal}. For $X = K$, Theorem~\ref{thm:dyn-set-PH-DYBE} reduces to the following form: Given a ternary operation $\blankternary$ on $\Lambda$, we define
\begin{equation}
  \label{eq:def-sigma-0}
  \sigma_0(\lambda)(x, y) = (\langle \lambda, x, y \rangle, y)
  \quad (\lambda \in \Lambda, x, y \in K).
\end{equation}
Then $\sigma_0 = \{ \sigma_0(\lambda): K \times K \to K \times K \}_{\lambda \in \Lambda}$ is a dynamical Yang-Baxter map on $K$ if and only if the ternary operation $\blankternary$ satisfies the equations~\eqref{eq:Shibu-ternary-1} and~\eqref{eq:Shibu-ternary-2} of Theorem~\ref{thm:dyn-set-PH-DYBE}. Moreover, every dynamical Yang-Baxter map on $K$ is obtained in this way from such a ternary operation.

Now let $X$ be a dynamical set over $\Lambda$ of PH type. Since there is an isomorphism $\psi_X: X \to K$ in $\DSet{\Lambda}$, the morphism $\sigma_0: K \otimes K \to K \otimes K$ induces a morphism $X \otimes X \to X \otimes X$. We express the induced morphism in an explicit way:

\begin{lemma}
  Let $\blankternary$ be a ternary operation on $\Lambda$, and let $\sigma_0$ be the morphism defined by~\eqref{eq:def-sigma-0}. Then the morphism
  \begin{equation*}
    \sigma = (\psi_X^{-1} \otimes \psi_X^{-1}) \diamond \sigma_0 \diamond (\psi_X \otimes \psi_X): X \otimes X \to X \otimes X
  \end{equation*}
  coincides with the morphism defined by the formulas~\eqref{eq:sigma-ternary-def-1}--\eqref{eq:sigma-ternary-def-4} by using the same ternary operation on $\Lambda$.
\end{lemma}

This lemma can be proved by the direct computation. To avoid technical computation in $\DSet{\Lambda}$, and to give a deeper understanding to the background algebraic materials, we rather prove this lemma from the viewpoint of the complete quivers corresponding to $X$ and $K$.

\begin{proof}
  Set $A = \mathsf{Q}(X)$ and $B = \mathsf{Q}(K)$. For $Q = A, B$, we denote by ``$\lambda \to \mu$ in $Q$'' the unique arrow of $Q$ starting from $\lambda$ ending at $\mu$. We extend this notation as in \eqref{eq:uniq-arr-notation} and have an obvious isomorphism
  \begin{equation*}
    \Psi: A \times_{\Lambda} A \to B \times_{\Lambda} B,
    \quad (\text{$\lambda \to \mu \to \nu$ in $A$})
    \mapsto (\text{$\lambda \to \mu \to \nu$ in $B$}).
  \end{equation*}
  We define $\widetilde{\sigma}: A \times_{\Lambda} A \to A \times_{\Lambda} A$ and $\widetilde{\sigma}_0: B \times_{\Lambda} B \to B \times_{\Lambda} B$ by
  \begin{equation*}
    \widetilde{\sigma} = (\mathsf{Q}^{(2)}_{X,X})^{-1} \circ \mathsf{Q}(\sigma) \circ \mathsf{Q}^{(2)}_{X,X}
    \quad \text{and} \quad
    \widetilde{\sigma}_0 = (\mathsf{Q}^{(2)}_{K,K})^{-1} \circ \mathsf{Q}(\sigma_0) \circ \mathsf{Q}^{(2)}_{K,K},
  \end{equation*}
  respectively. Now we consider the following diagram:
  \begin{equation*}
    \xymatrix@C=40pt{
      \mathsf{Q}(X \otimes X)
      \ar[d]_{\mathsf{Q}(\sigma)}
      \ar[r]^(.55){(\mathsf{Q}^{(2)}_{X,X})^{-1}}
      & A \times_{\Lambda} A
      \ar[d]_{\widetilde{\sigma}}
      \ar[r]^{\Psi}
      & B \times_{\Lambda} B
      \ar[d]^{\widetilde{\sigma}_0}
      \ar[r]^(.45){\mathsf{Q}^{(2)}_{K,K}}
      & \mathsf{Q}(K \otimes K)
      \ar[d]^{\mathsf{Q}(\sigma_0)} \\
      \mathsf{Q}(X \otimes X)
      \ar[r]^(.55){(\mathsf{Q}^{(2)}_{X,X})^{-1}}
      & A \times_{\Lambda} A
      \ar[r]^{\Psi}
      & B \times_{\Lambda} B
      \ar[r]^(.45){\mathsf{Q}^{(2)}_{K,K}}
      & \mathsf{Q}(K \otimes K)
    }
  \end{equation*}
  The left and the right squares commute. We use the rule~\eqref{eq:attaching-rule} to compute
  \begin{align*}
    \widetilde{\sigma}(\text{$\lambda \to \mu \to \nu$ in $A$})
    & = (\text{$\lambda \to \langle \lambda, \mu, \nu \rangle \to \nu$ in $A$}), \\
    \widetilde{\sigma}_0(\text{$\lambda \to \mu \to \nu$ in $B$})
    & = (\text{$\lambda \to \langle \lambda, \mu, \nu \rangle \to \nu$ in $B$})
  \end{align*}
  for $\lambda, \mu, \nu \in \Lambda$. Hence the middle square of the above diagram also commutes. Now we express $\Psi$ explicitly: For $\lambda, \mu \in \Lambda$ and $x, y \in X$ with $\mu = \lambda \dyact x$, we have
  \begin{align*}
    \Psi((\lambda, x), (\mu, y))
    & = \Psi(\text{$\lambda \to \mu \to (\mu \dyact y)$ in $A$}) \\
    & = (\text{$\lambda \to \mu \to (\mu \dyact y)$ in $B$}) \\
    & = ((\lambda, \mu), (\mu, \mu \dyact y)) \\
    & = ((\lambda, \lambda \dyact x), (\mu, \mu \dyact y)).
  \end{align*}
  Namely, $\Psi = \mathsf{Q}(\psi_X) \times_{\Lambda} \mathsf{Q}(\psi_X)$. By the functorial property of $\mathsf{Q}^{(2)}$,
  \begin{equation*}
    \mathsf{Q}^{(2)}_{K,K} \circ \Psi \circ (\mathsf{Q}^{(2)}_{X,X})^{-1}
    = \mathsf{Q}(\psi_X \otimes \psi_X) \circ \mathsf{Q}^{(2)}_{X,X} \circ (\mathsf{Q}^{(2)}_{X,X})^{-1}
    = \mathsf{Q}(\psi_X \otimes \psi_X).
  \end{equation*}
  Hence, by the above commutative diagram, we have
  \begin{equation*}
    \mathsf{Q}(\sigma) \circ \mathsf{Q}(\psi_X \otimes \psi_X)
    = \mathsf{Q}(\psi_X \otimes \psi_X) \circ \mathsf{Q}(\sigma_0).
  \end{equation*}
  Now the claim of this theorem follows from the faithfulness of $\mathsf{Q}$.
\end{proof}

\begin{proof}[Proof of Theorem~\ref{thm:DYB-map-PH-equiv}]
  The proof boils down to the case where $X_1 = X_2 = K$ by the above lemma. If $X_1 = X_2 = K$, the claim easily follows from the fact that there are no morphisms other than the identity morphism.
\end{proof}

\subsection{Classification of some subclasses of the solutions}

Let $X$ be a dynamical set over $\Lambda$ of PH type. Our method is effective to characterize several classes of dynamical Yang-Baxter maps on $X$. Let $\blankternary$ be a ternary operation on $\Lambda$ satisfying \eqref{eq:Shibu-ternary-1} and~\eqref{eq:Shibu-ternary-2}, and let $\sigma: X \otimes X \to X \otimes X$ be the dynamical Yang-Baxter map arising from the ternary operation $\blankternary$.

\begin{theorem}
  \label{thm:DYB-map-PH-subclass}
  {\rm (a)} $\sigma$ is unitary \textup{(}i.e., $\sigma \diamond \sigma$ is the identity\textup{)} if and only if
  \begin{equation*}
    \langle a, \langle a, b, c \rangle, c \rangle = b
    \quad (\forall a, b, c \in \Lambda).
  \end{equation*}
  {\rm (b)} $\sigma$ is idempotent \textup{(}i.e., $\sigma \diamond \sigma = \sigma$\textup{)} if and only if
  \begin{equation*}
    \langle a, \langle a, b, c \rangle, c \rangle
    = \langle a, b, c \rangle
    \quad (\forall a, b, c \in \Lambda).
  \end{equation*}
  {\rm (c)} $\sigma$ is invertible if and only if the following map is bijective for all $a, b \in \Lambda$\textup{:}
  \begin{equation*}
    \Lambda \to \Lambda, 
    \quad x \mapsto \langle a, x, c \rangle.
  \end{equation*}
\end{theorem}

Part (a) of this theorem is \cite[Proposition 7.1]{MR2389797}.

\begin{proof}
  Let $\widetilde{\sigma}: \mathsf{Q}(X) \times_{\Lambda} \mathsf{Q}(X) \to \mathsf{Q}(X) \times_{\Lambda} \mathsf{Q}(X)$ be the morphism in $\Quiv{\Lambda}$ corresponding to $\sigma$ via~\eqref{eq:equiv-X2-A2}. Then the dynamical Yang-Baxter map $\sigma$ is unitary if and only if $\widetilde{\sigma} \circ \widetilde{\sigma}$ is the identity. Part (a) follows from
  \begin{equation*}
    \widetilde{\sigma}\widetilde{\sigma}(a \to b \to c)
    = a \to \langle a, \langle a, b, c \rangle, c \rangle \to c
    \quad (a, b, c \in \Lambda).
  \end{equation*}
  Part (b) is proved by the same computation. Part (c) follows from the fact that $\sigma$ is invertible if and only if $\widetilde{\sigma}$ is.
\end{proof}

\subsection{Solutions of vertex type}

A dynamical Yang-Baxter map $(X, \sigma)$ is of {\em vertex type} \cite{MR2389797} if the map $\sigma(\lambda)$ does not depend on the parameter $\lambda \in \Lambda$. We note that the class of the dynamical Yang-Baxter maps of vertex type is {\em not} closed under the equivalence, as the following example shows:

\begin{example}[Shibukawa {\cite[Section 8]{MR2389797}}]
  \label{ex:Shibu-vertex-type}
  Let $*_i$ ($i = 1, 2$) be a binary operation on $\Lambda$ such that $(\Lambda, *_i)$ is a left quasigroup with left division $/_i$. We define the dynamical set $X_i$ of PH type to be the set $X_i = \Lambda$ equipped with the structure map $*_i: \Lambda \times X_i \to \Lambda$. Suppose that the first operation $*_1$ satisfies
  \begin{equation}
    \label{eq:ternary-1-condition}
    (a *_1 c) \backslash_1 ((a *_1 b) *_1 c)
    = (a' *_1 c) \backslash_1 ((a' *_1 b) *_1 c)
  \end{equation}
  for all $a, a', b, c \in \Lambda$ (this condition is satisfied if, for example, $\Lambda$ is a group with respect to $*_1$). We require nothing of $*_2$. The ternary operation
  \begin{equation*}
    \langle a, b, c \rangle_1 := a *_1 (b \backslash_1 c)
    \quad (a, b, c \in \Lambda)
  \end{equation*}
  satisfies~\eqref{eq:Shibu-ternary-1} and~\eqref{eq:Shibu-ternary-2}. Now we define $\sigma_i: X_i \otimes X_i \to X_i \otimes X_i$ ($i = 1, 2$) by the formulas~\eqref{eq:sigma-ternary-def-1}--\eqref{eq:sigma-ternary-def-4} with $X = X_i$, $\dyact = *_i$ and $\blankternary = \blankternary_1$. Theorem~\ref{thm:dyn-set-PH-DYBE} says that $(X_1, \sigma_1)$ and $(X_2, \sigma_2)$ are dynamical Yang-Baxter maps of PH type. The former is given by
  \begin{equation*}
    \sigma_1(\lambda)(x, y) = (y, (\lambda *_1 y) \backslash_1 ((\lambda *_1 x) *_1 y)
    \quad (\lambda \in \Lambda, x, y \in X_1),
  \end{equation*}
  which is a vertex type by virtue of~\eqref{eq:ternary-1-condition}. The latter is given by
  \begin{align*}
    \sigma_2(\lambda)(x, y)
    = (\lambda \backslash_2
    & \, (\lambda *_1 ((\lambda *_2 x) \backslash_1 ((\lambda *_2 x) *_2 y))), \\
    & \, (\lambda *_1 ((\lambda *_2 x) \backslash_1 ((\lambda *_2 x) *_2 y))) \backslash_2 ((\lambda *_2 x) *_2 y))
  \end{align*}
  for $\lambda \in \Lambda$ and $x, y \in X$, which is not a vertex type in general (see \cite{MR2389797} for a concrete example).
\end{example}

Following Shibukawa \cite{MR2389797},  we say that {\em a dynamical Yang-Baxter map $(X, \sigma)$ has a vertex-IRF correspondence} if it is equivalent to a dynamical Yang-Baxter map of vertex type. It is difficult to determine whether a given dynamical Yang-Baxter map has a vertex-IRF correspondence. For future study of this kind of problems, we give the following new examples of dynamical Yang-Baxter maps having a vertex-IRF correspondence.

\begin{example}
  \newcommand{\ope}{\mathbin{\sharp}}
  Let $\ope$ be a binary operation on $\Lambda$ satisfying
  \begin{equation}
    \label{eq:nazo}
    b \ope (c \ope d) = (b \ope c) \ope (c \ope d)
    \quad \text{and} \quad
    c \ope d = (c \ope d) \ope d
  \end{equation}
  for all $b, c, d \in \Lambda$ (examples of such a binary operation will be given below). Since the ternary operation $\langle a, b, c \rangle = b \ope c$ ($a, b, c \in \Lambda$) satisfies~\eqref{eq:Shibu-ternary-1} and~\eqref{eq:Shibu-ternary-2},
  \begin{equation*}
    \sigma_0: K \otimes K \to K \otimes K,
    \quad \sigma_0(\lambda)(x, y) = (x \ope y, y)
    \quad (\lambda \in \Lambda, x, y \in K)
  \end{equation*}
  is a dynamical Yang-Baxter map of vertex type on the dynamical set $K = K_{\Lambda}$ of Lemma~\ref{lem:dyn-set-terminal}. Now we suppose that we are given a binary operation $*$ on $\Lambda$ such that $(\Lambda, *)$ is a left quasigroup. Let $X = \Lambda$ be the dynamical set with the structure map given by $\dyact_X = *$. We have a dynamical Yang-Baxter map
  \begin{align*}
    \sigma(\lambda)(x, y) = (\lambda \backslash_X
    & ((\lambda * x) \ope ((\lambda * x) * y)), \\
    & ((\lambda * x) \ope ((\lambda * x) * y)) \backslash_X ((\lambda * x) * y))
      \quad (\lambda \in \Lambda, x, y \in X)
  \end{align*}
  on $X$, which is equivalent to $\sigma_0$ by Theorem~\ref{thm:DYB-map-PH-equiv}. Unlike Example~\ref{ex:Shibu-vertex-type}, $\sigma$ is always idempotent by Theorem~\ref{thm:DYB-map-PH-subclass}.

  The problem is to find a binary operation satisfying~\eqref{eq:nazo}. A {\em band}, introduced by Clifford \cite{MR0062119}, is a set $S$ endowed with an associative binary operation $\vee$ such that $a \vee a = a$ for all elements $a \in S$. It is easy to verify that the binary operation of a band satisfies \eqref{eq:nazo}. A familiar example of bands may be a {\em semilattice}, which is an abstraction of a partially ordered set having a least upper bound of any pair of its elements. Algebraically, a semilattice is just a band $(S, \vee)$ such that $a \vee b = b \vee a$ for all $a, b \in S$.

  For example, the set $\Lambda = \mathbb{R}$ of all real numbers is a semilattice by the binary operation $\vee$ defined by $a \vee b = \max\{ a, b \}$. We use this semilattice to claim that the above construction produces an idempotent dynamical Yang-Baxter map $\sigma = \{ \sigma(\lambda) \}_{\lambda \in \Lambda}$ depending on the parameter $\lambda \in \Lambda$. We make the set $X = \Lambda$ ($= \mathbb{R}$) a left quasigroup by the binary operation
  \begin{equation*}
    x * y
    := \text{(the middle point of $x$ and $y$)}
    = \frac{x + y}{2} \quad (x, y \in X).
  \end{equation*}
  To compute $\sigma(\lambda)$ with $\ope = \vee$, we shall note
  \begin{equation*}
    (\lambda * x) * y - \lambda * x
    = \frac{1}{2}y - \frac{1}{4}\lambda - \frac{1}{4}x
    = \frac{1}{2}(y - \lambda * x).
  \end{equation*}
  From this equation, we have
  \begin{equation*}
    ((\lambda * x) \vee ((\lambda * x) * y)) =
    \begin{cases}
      \frac{1}{4} \lambda + \frac{1}{4} x + \frac{1}{2} y
      & \text{if $y \ge \lambda * x$}, \\
      \frac{1}{2} \lambda + \frac{1}{2} x
      & \text{otherwise}.
    \end{cases}
  \end{equation*}
  The left division of $(X, *)$ is given by $a \backslash b = 2 b - a$ for $a, b \in X$. Thus,
  \begin{equation*}
    \sigma(\lambda)(x, y)
    = \left(- \frac{1}{2} \lambda + \frac{1}{2} x + y,
      \, \frac{1}{4} \lambda + \frac{1}{4} x + \frac{1}{2} y \right)
  \end{equation*}
  for $x, y \in X$ if $y \ge \lambda * x$, and $\sigma(\lambda)(x, y) = (x, y)$ otherwise.
\end{example}

\section{A new class of dynamical Yang-Baxter maps}
\label{sec:new-sol}

\subsection{A new class of dynamical Yang-Baxter maps}

In this section, we give a new class of dynamical Yang-Baxter maps by using the quiver-theoretical techniques developed in the above. Let $\Lambda$ be a non-empty set. We first prepare the dynamical set $X \in \DSet{\Lambda}$ underlying the dynamical Yang-Baxter map that we will construct. We fix a non-empty set $X_0$ and a family $\{ *_i \}_{i \in X_0}$ of binary operations on $\Lambda$ indexed by $X_0$ such that $(\Lambda, *_i)$ is a left quasigroup for all $i \in X_0$. We set $X = X_0 \times \Lambda$ and make it a dynamical set over $\Lambda$ by the structure map defined by
\begin{equation*}
  \dyact_X: \Lambda \times X \to \Lambda,
  \quad \lambda \dyact (i, \mu) = \lambda *_i \mu \quad (\lambda, \mu \in \Lambda, i \in X_0).
\end{equation*}
Our construction requires the following two data:
\begin{itemize}
\item A map $\sigma_0: X_0 \times X_0 \to X_0 \times X_0$.
\item A ternary operation $\blankternary$ on $\Lambda$.
\end{itemize}
For $i, j \in X_0$, we write
$\sigma_0(i, j) = (i \triactl j, i \triactr j)$. We denote by $\backslash_i$ the left division of the left quasigroup $(\Lambda, *_i)$. We now define $\sigma: X \otimes X \to X \otimes X$ in $\DSet{\Lambda}$ by
\begin{equation}
  \label{eq:new-sigma-def-1}
  \sigma(\lambda)((i, x), (j, y))
  = \Big( (i, x) \dybml{\lambda} (j, y), (i, x) \dybmr{\lambda} (j, y) \Big)
\end{equation}
for $(i, x), (j, y) \in X = X_0 \times \Lambda$, where
\begin{align}
  \label{eq:new-sigma-def-2}
  (i, x) \dybml{\lambda} (j, y)
  & = \Big( i \triactl j, \lambda
    \big\backslash_{i \triactl j}
    \langle \lambda, \lambda *_i x, (\lambda *_i x) *_j y \rangle \Big), \\
  \label{eq:new-sigma-def-3}
  (i, x) \dybmr{\lambda} (j, y)
  & = \Big( i \triactr j, \langle \lambda, \lambda *_i x, (\lambda *_i x) *_j y \rangle
    \big\backslash_{i \triactr j}
    ((\lambda *_i x) *_j y) \Big).
\end{align}
The main theorem of this section is:

\begin{theorem}
  \label{sec:sec-4-main-thm}
  The morphism $\sigma = \{ \sigma(\lambda) \}$ is a dynamical Yang-Baxter map on $X$ if and only if $\sigma_0$ is a Yang-Baxter map on $X_0$ and the ternary operation $\blankternary$ satisfies \eqref{eq:Shibu-ternary-1} and \eqref{eq:Shibu-ternary-2} of Theorem~\ref{thm:dyn-set-PH-DYBE}.
\end{theorem}

Thus, in a sense, one can obtain a dynamical Yang-Baxter map on $X$ by composing a dynamical Yang-Baxter map on $X_0$ and a dynamical Yang-Baxter map arising from a ternary operation on $\Lambda$.

Theorem~\ref{thm:dyn-set-PH-DYBE} is the case where $X_0$ is a singleton. Theorem~\ref{thm:DYB-map-PH-subclass} classifies some classes of dynamical Yang-Baxter maps on a dynamical set of PH type in terms of the associated ternary operator. A similar result holds for the dynamical Yang-Baxter maps constructed by the above theorem.

\begin{theorem}
  \label{sec:sec-4-main-thm-suppl}
  Suppose that $\sigma_0$ is a Yang-Baxter map on $X_0$ and the ternary operation $\blankternary$ on $\Lambda$ satisfies \eqref{eq:Shibu-ternary-1} and \eqref{eq:Shibu-ternary-2} of Theorem~\ref{thm:dyn-set-PH-DYBE} so that the morphism $\sigma$ defined by~\eqref{eq:new-sigma-def-1}--\eqref{eq:new-sigma-def-3} is a dynamical Yang-Baxter map. Then there holds:
  \begin{itemize}
  \item [(a)] $\sigma$ is unitary if and only if $\sigma_0$ is unitary and
    \begin{equation*}
      \langle a, \langle a, b, c \rangle, c \rangle = b
      \quad (\forall a, b, c \in \Lambda).
    \end{equation*}
  \item [(b)] $\sigma$ is idempotent if and only if $\sigma_0$ is idempotent and
    \begin{equation*}
      \langle a, \langle a, b, c \rangle, c \rangle
      = \langle a, b, c \rangle
      \quad (\forall a, b, c \in \Lambda).
    \end{equation*}
  \item [(c)] $\sigma$ is invertible if and only if $\sigma_0$ is invertible and the map
    \begin{equation*}
      \Lambda \to \Lambda, 
      \quad x \mapsto \langle a, x, c \rangle
    \end{equation*}
    is bijective for all $a, b \in \Lambda$.
  \end{itemize}
\end{theorem}

\subsection{Proof of Theorems~\ref{sec:sec-4-main-thm} and~\ref{sec:sec-4-main-thm-suppl}}

We keep the notation as above. To prove Theorem~\ref{sec:sec-4-main-thm}, we consider the quiver $A \in \Quiv{\Lambda}$ defined by
\begin{equation*}
  A = \Lambda \times X_0 \times \Lambda,
  \quad \src_A(\lambda, i, \mu) = \lambda
  \quad \text{and} \quad \tgt_A(\lambda, i, \mu) = \mu.
\end{equation*}
We define $K \in \DSet{\Lambda}$ to be the set $K = X_0 \times \Lambda$ with the structure map given by $\lambda \dyact_K (i, \mu) = \mu$. The quiver $A$ is naturally identified with $\mathsf{Q}(K)$. The dynamical set $X$ introduced in the above is isomorphic to $K$ ({\it cf}. Lemma~\ref{lem:dyn-set-PH-unique}). More precisely, there holds:

\begin{lemma}
  \label{lem-new-sigma-lem-1}
  For $\lambda \in \Lambda$, we define the map $\psi(\lambda): X \to K$ by
  \begin{equation*}
    \psi(\lambda)(i, \lambda') = (i, \lambda *_i \lambda')
    \quad (\lambda, \lambda' \in \Lambda, i \in X_0).
  \end{equation*}
  Then $\psi = \{ \psi(\lambda) \}_{\lambda \in \Lambda}$ is an isomorphism $\psi: X \to K$ in $\DSet{\Lambda}$ with inverse
  \begin{equation*}
    \psi^{-1}: K \to X,
    \quad \psi^{-1}(\lambda)(i, \lambda') = (i, \lambda \backslash_i \lambda')
    \quad (\lambda, \lambda' \in \Lambda, i \in X_0).
  \end{equation*}  
\end{lemma}

This result implies that the dynamical Yang-Baxter maps on $X$ are in one-to-one correspondence to the set of morphisms $\widetilde{\sigma}: A \times_{\Lambda} A \to A \times_{\Lambda} A$ in $\Quiv{\Lambda}$ such that $(A, \widetilde{\sigma})$ is a braided quiver. Now let $\widetilde{\sigma}: A \times_{\Lambda} A \to A \times_{\Lambda} A$ be a morphism in $\Quiv{\Lambda}$. For simplicity, we write an arrow $(\lambda, i, \mu) \in A$ as $\lambda \xrightarrow{\ i \ } \mu$ (the arrows are ``colored'' by $X_0$ unlike the setting of the previous section). We extend this notation as in~\eqref{eq:uniq-arr-notation}. Then, since
\begin{equation*}
  A \times_{\Lambda} A = \{ \lambda \xrightarrow{\ i \ } \mu \xrightarrow{\ j \ } \nu \mid \lambda, \mu, \nu \in \Lambda, i, j \in X_0 \},
\end{equation*}
one can express the morphism $\widetilde{\sigma}$ as
\begin{equation}
  \label{eq:yabai-map}
  \widetilde{\sigma}(\lambda \xrightarrow{\ i \ } \mu \xrightarrow{\ j \ } \nu)
  = \lambda \xrightarrow{\ t_1(\lambda, i, \mu, j, \nu) \ }
  t_2(\lambda, i, \mu, j, \nu)
  \xrightarrow{\ t_3(\lambda, i, \mu, j, \nu) \ } \nu
\end{equation}
by using a triple $(t_1: M \to X_0, t_2: M \to \Lambda, t_3: M \to X_0)$ of maps from $M$, where $M = \Lambda \times X_0 \times \Lambda \times X_0 \times \Lambda$. In other words:

\begin{lemma}[{\it cf}. Lemma~\ref{lem:mor-ternary}]
  \label{lem-new-sigma-lem-2}
  There is a bijection
  \begin{equation*}
    \Quiv{\Lambda}(A \times_{\Lambda} A, A \times_{\Lambda} A)
    \cong \Map(M, X_0) \times \Map(M, \Lambda) \times \Map(M, X_0).
  \end{equation*}
\end{lemma}

One can therefore, in principle, write down the necessary and sufficient condition for $\widetilde{\sigma}$ to satisfy the braid relation in terms of the triple $(t_1, t_2, t_3)$. The result would be inconvenient and far from understandable. For this reason, we make an ansatz that the maps $t_1$, $t_2$ and $t_3$ are expressed by the map $\sigma_0: X_0 \times X_0 \to X_0 \times X_0$ and the ternary operation $\blankternary$ on $\Lambda$ in the following way:
\begin{equation*}
  (t_1(\lambda, i, \mu, j, \nu), t_3(\lambda, i, \mu, j, \nu)) = \sigma_0(i, j),
  \quad t_2(\lambda, i, \mu, j, \nu) = \langle \lambda, \mu, \nu \rangle.
\end{equation*}
Then the formula \eqref{eq:yabai-map} reduces to the following form:
\begin{equation}
  \label{eq:yabakunai}
  \widetilde{\sigma}(\lambda \xrightarrow{\ i \ } \mu \xrightarrow{\ j \ } \nu)
  = \lambda \xrightarrow{\ i \triactl j \ }
  \langle \lambda, \mu, \nu \rangle
  \xrightarrow{\ i \triactr j \ } \nu.
\end{equation}
The morphisms $\widetilde{\sigma}_{12} \widetilde{\sigma}_{23} \widetilde{\sigma}_{12}$ and $\widetilde{\sigma}_{23} \widetilde{\sigma}_{12} \widetilde{\sigma}_{23}$ can be computed in a similar way as the proof of Theorem~\ref{thm:dyn-set-PH-DYBE}; see Subsection \ref{subsec:proof-thm-33}. As a consequence, we have:

\begin{lemma}
  \label{lem-new-sigma-lem-3}
  Suppose that $\widetilde{\sigma}$ is given by~\eqref{eq:yabakunai}. Then $(A, \widetilde{\sigma})$ is a braided quiver if and only if $\sigma_0$ is a Yang-Baxter map and the ternary operation $\blankternary$ satisfies \eqref{eq:Shibu-ternary-1} and \eqref{eq:Shibu-ternary-2} of Theorem~\ref{thm:dyn-set-PH-DYBE}.
\end{lemma}

Now we give a proof of Theorem~\ref{sec:sec-4-main-thm}.

\begin{proof}[Proof of Theorem~\ref{sec:sec-4-main-thm}]
  Let $\sigma: X \otimes X \to X \otimes X$ be the morphism in $\DSet{\Lambda}$ defined by~\eqref{eq:new-sigma-def-1}--\eqref{eq:new-sigma-def-3} from the map $\sigma_0: X_0 \times X_0 \to X_0 \times X_0$ and the ternary operation $\blankternary$ on $\Lambda$. We consider the diagram
  \begin{equation*}
    \xymatrix@C=64pt{
      \mathsf{Q}(X \otimes X)
      \ar[d]_{\mathsf{Q}(\sigma)}
      \ar[r]^(.45){(\mathsf{Q}^{(2)}_{X,X})^{-1}}
      & \mathsf{Q}(X) \times_{\Lambda} \mathsf{Q}(X)
      \ar[r]^(.47){\mathsf{Q}(\psi) \times_{\Lambda} \mathsf{Q}(\psi)\mathstrut}
      & \mathsf{Q}(K) \times_{\Lambda} \mathsf{Q}(K)\phantom{,}
      \ar[d]^{\widetilde{\sigma}} \\
      \mathsf{Q}(X \otimes X)
      \ar[r]^(.45){(\mathsf{Q}^{(2)}_{X,X})^{-1}}
      & \mathsf{Q}(X) \times_{\Lambda} \mathsf{Q}(X)
      \ar[r]^(.47){\mathsf{Q}(\psi) \times_{\Lambda} \mathsf{Q}(\psi)\mathstrut}
      & \mathsf{Q}(K) \times_{\Lambda} \mathsf{Q}(K),
    }
  \end{equation*}
  where $\psi: X \to K$ is the isomorphism given by Lemma~\ref{lem-new-sigma-lem-1} and $\widetilde{\sigma}$ is the morphism defined by~\eqref{eq:yabakunai}. By Lemma~\ref{lem-new-sigma-lem-3}, it is sufficient to show that this diagram is commutative to prove this theorem. For simplicity, we set
  \begin{equation*}
    \Xi = (\mathsf{Q}(\psi) \times_{\Lambda} \mathsf{Q}(\psi)) \circ (\mathsf{Q}^{(2)}_{X,X})^{-1}.
  \end{equation*}
  For all $\lambda, x, y \in \Lambda$ and $i, j \in X_0$, we compute
  \begin{align*}
    \Xi(\lambda, (i, x), (j, y))
    & = ((\mathsf{Q}(\psi) \times_{\Lambda} \mathsf{Q}(\psi)) ((\lambda, i, x)), (\lambda *_i x, j, y))) \\
    & = ((\lambda, i, \lambda *_i x), (\lambda *_i x, j, (\lambda *_i x) *_j y)) \\
    & = \lambda \xrightarrow{\ i \ } \lambda *_i x \xrightarrow{\ j \ } (\lambda *_i x) *_j y.
  \end{align*}
  By the definition of $\sigma$ and $\widetilde{\sigma}$, we have
  \begin{align*}
    \Xi\mathsf{Q}(\sigma)(\lambda, (i, x), (j, y))
    & = \Xi (\lambda, (i,x) \dybml{\lambda} (j,y), (i,x) \dybmr{\lambda} (j,y)) \\
    & = \lambda \xrightarrow{\ i \triactl j \ }
      \langle \lambda, \lambda *_i x, (\lambda *_i x) *_j y \rangle
      \xrightarrow{\ i \triactr j \ } (\lambda *_i x) *_j y \\
    & = (\widetilde{\sigma} \circ \xi) (\lambda, (i, x), (j, y)).
  \end{align*}
  The proof is done.
\end{proof}

\begin{proof}[Proof of Theorem~\ref{sec:sec-4-main-thm-suppl}]
  This theorem is proved by the same way as Theorem~\ref{thm:DYB-map-PH-subclass} but by using \eqref{eq:yabakunai} instead of \eqref{eq:attaching-rule}.
\end{proof}

\section{Weak bialgebras arising from a solution}
\label{sec:two-weak-bialg}

\newcommand{\E}[2]{\mathop{%
    \renewcommand{\arraystretch}{.85}%
    \setlength{\arraycolsep}{1pt}%
    \mathbf{e} \! \left[
      \begin{array}{c} {\mathstrut #1} \\ {\mathstrut #2} \end{array}
    \right]}}
\newcommand{\facewt}[5]{\mathop{%
    \setlength{\arraycolsep}{1pt}%
    \renewcommand{\arraystretch}{.75}%
    {#1} \! \left[
      \begin{array}{ccc} & {\mathstrut #2} \\ {\mathstrut #4} & & {\mathstrut #3} \\ & {\mathstrut #5} \end{array}
    \right]}}
\newcommand{\W}[4]{\facewt{\mathbf{w}}{#1}{#2}{#3}{#4}}

\subsection{Weak bialgebras arising from a solution}

Throughout this section, we work over a fixed field $k$. By an algebra, we always mean an associative and unital algebra over the field $k$. The term `coalgebra' is used in a similar manner. A {\em weak bialgebra} is an algebra endowed with a coalgebra structure $\Delta: B \to B \otimes_k B$ and $\varepsilon: B \to k$ such that
\begin{gather*}
  \Delta(a b) = \Delta(a) \Delta(b), \\
  (\Delta(1) \otimes 1) \cdot (1 \otimes \Delta(1))
  = 1_{(1)} \otimes 1_{(2)} \otimes 1_{(3)}
  = (1 \otimes \Delta(1)) \cdot (\Delta(1) \otimes 1), \\
  \varepsilon(a b_{(2)}) \varepsilon(b_{(1)} c)
  = \varepsilon(a b c) = \varepsilon(a b_{(1)}) \varepsilon(b_{(2)} c)
\end{gather*}
for all $a, b, c \in B$. Here we have used the Sweedler notation, such as
\begin{equation*}
  \Delta(a) = a_{(1)} \otimes a_{(2)}
  \quad \text{and} \quad
  \Delta(a_{(1)}) \otimes a_{(2)}
  = a_{(1)} \otimes a_{(2)} \otimes a_{(3)}
  = a_{(1)} \otimes \Delta(a_{(2)})
\end{equation*}
to express the comultiplication of $a \in B$.

Let $\Lambda$ be a non-empty set, and let $\mathcal{F} = \Map(\Lambda, k)$ be the algebra of the $k$-valued functions on $\Lambda$. On the one hand, Shibukawa \cite{MR3448180} introduced a construction of a bialgebroid over $\mathcal{F}$ ($=$ $\times_{\mathcal{F}}$-bialgebra in the sense of Sweedler \cite{MR0364332} and Takeuchi \cite{MR0506407}) from a dynamical Yang-Baxter map $(X, \sigma)$ satisfying the following condition:
\begin{equation}
  \label{eq:Shibukawa-condition}
  \begin{gathered}
    \text{the map $\sigma(\lambda)$ is invertible for all $\lambda \in \Lambda$, the set $X$ is finite}, \\[-3pt]
    \text{and the map $\Lambda \to \Lambda$ given by $\lambda \mapsto \lambda \dyact x$ is bijective for all $x \in X$}.
  \end{gathered}
\end{equation}
On the other hand, Hayashi \cite{MR1623965} gave a construction of a weak bialgebra from a certain class of braided object of the category $\bimod{\mathcal{F}}$ of $\mathcal{F}$-bimodules. Thus, when $\Lambda$ is finite, one can obtain a weak bialgebra by applying his construction to the linearization of the braided quiver
\begin{equation*}
  (Q, \widetilde{\sigma}) := \Br(\mathsf{Q})(X, \sigma).
\end{equation*}
Suppose that $\Lambda$ is finite. Then $\mathcal{F}$ is a separable algebra and hence a bialgebroid over $\mathcal{F}$ is a weak bialgebra by Schauenburg \cite{MR2024429}. Thus we have two constructions of a weak bialgebra from a dynamical Yang-Baxter map satisfying~\eqref{eq:Shibukawa-condition}. The aim of this section is to discuss relations between these two constructions.

\subsection{Hayashi's construction}

We fix a non-empty finite set $\Lambda$ and consider the algebra $\mathcal{F} = \Map(\Lambda, k)$ of $k$-valued functions on $\Lambda$. If $Q$ is a quiver over $\Lambda$, then the vector space $W = \mathrm{span}_k(Q)$ spanned by $Q$ is an $\mathcal{F}$-bimodule by
\begin{equation*}
  f \cdot a = f(\src(a)) a
  \quad \text{and} \quad
  a \cdot f = f(\tgt(a)) a
\end{equation*}
for $f \in \mathcal{F}$ and $a \in Q$. A {\em star-triangular face model} is a pair $(Q, w)$ consisting of a finite quiver $Q$ and a morphism $w: W \otimes_{\mathcal{F}} W \to W \otimes_{\mathcal{F}} W$ of $\mathcal{F}$-bimodules, where $W = \mathrm{span}_k(Q)$ is the $\mathcal{F}$-bimodule associated to $Q$, such that $w$ is invertible and $(W, w)$ is a braided object of $\bimod{\mathcal{F}}$. Hayashi \cite{MR1623965} introduced a construction of a weak bialgebra $\mathfrak{A}(w)$ from such a model and studied its Hopf closure, that is, a weak Hopf algebra with a certain universal property for $\mathfrak{A}(w)$.

We recall the construction of $\mathfrak{A}(w)$. Let $(Q, w)$ be a star-triangular face model, and set $W = \mathrm{span}_k(Q)$. For a positive integer $m$, we define $Q^{(m)}$ inductively by
\begin{equation*}
  Q^{(1)}= Q \text{\quad and \quad} Q^{(m + 1)} = Q^{(m)} \times_{\Lambda} Q.
\end{equation*}
An element of $Q^{(m)}$ can be regarded as a path on $Q$ of length $m$. For simplicity of notation, we set $Q^{(0)} = \Lambda$. An element $\lambda \in Q^{(0)}$ will be regarded as a path on $Q$ of length zero from $\lambda$ to $\lambda$.

It is easy to see that the set $\{ a \otimes b \in W \otimes_{\mathcal{F}} W \mid (a, b) \in Q^{(2)} \}$ is a basis of the $k$-vector space $W \otimes_{\mathcal{F}} W$. We say that a quadruple $(a, b, c, d)$ of elements of $Q$ form a {\em face} if $\src(a) = \src(c)$, $\tgt(a) = \src(b)$, $\tgt(c) = \src(d)$, $\tgt(b) = \tgt(d)$. Since $w$ is a morphism of $\mathcal{F}$-bimodules, we have
\begin{equation*}
  w(a \otimes b) \in \mathrm{span}_k \{ c \otimes d \mid \text{$c, d \in Q$ and $(a, b, c, d)$ form a face} \}
\end{equation*}
for all $(a, b) \in Q^{(2)}$. We define $\W{a}{b}{c}{d} \in k$ for a face $(a, b, c, d)$ by
\begin{equation*}
  w(a \otimes b) = \sum_{(c, d) \in Q^{(2)}} \W{a}{b}{c}{d} c \otimes d
  \quad ((a, b) \in Q^{(2)}).
\end{equation*}
For simplicity, we set $\W{a}{b}{c}{d} = 0$ unless $(a, b, c, d)$ form a face.

\begin{definition}
  The weak bialgebra $\mathfrak{A}(w)$ associated to the star-triangular face model $(W, w)$ is, as a $k$-algebra, generated by the symbols
  \begin{equation*}
    \E{p}{q} \quad (p, q \in Q^{(m)}, m = 0, 1, 2, \dotsc)
  \end{equation*}
  subject to the following relations:
  \begin{equation}
    \label{eq:Hayashi-rel-0}
    \sum_{\lambda, \mu \in Q^{(0)}} \E{\lambda}{\mu} = 1,
  \end{equation}
  \begin{equation}
    \label{eq:Hayashi-rel-1}
    \E{p}{q} \cdot \E{p'}{q'} = \delta_{\tgt(p), \src(p')} \delta_{\tgt(q), \src(q')} \E{p \, p'}{q \, q'}
  \end{equation}
  for $p, q \in Q^{(m)}$ and $p', q' \in Q^{(n)}$ ($m, n = 0, 1, 2, \dotsc$), and
  \begin{equation}
    \label{eq:Hayashi-rel-2}
    \sum_{(x, y) \in Q^{(2)}} \W{x}{y}{a}{b} \E{x}{c} \E{y}{d}
    = \sum_{(x, y) \in Q^{(2)}} \W{c}{d}{x}{y} \E{a}{x} \E{b}{y}
  \end{equation}
  for $(a, b), (c, d) \in Q^{(2)}$. The coalgebra structure of $\mathfrak{A}(w)$ is determined by
  \begin{equation*}
    \Delta \left( \E{p}{q} \right) = \sum_{t \in Q^{(m)}} \E{p}{t} \otimes \E{t}{q}
    \quad \text{and} \quad \varepsilon \left( \E{p}{q} \right) = \delta_{p,q}.
  \end{equation*}
  for $p, q \in Q^{(m)}$ ($m = 0, 1, 2, \dotsc$).
\end{definition}

\subsection{Shibukawa's construction}

Let $\Lambda$ be a non-empty set, and let $\mathcal{F}$ be the algebra of $k$-valued functions on $\Lambda$. Shibukawa \cite{MR3448180} introduced a method to construct a bialgebroid over $\mathcal{F}$ from a dynamical Yang-Baxter map satisfying~\eqref{eq:Shibukawa-condition}. To describe his construction in a short way, we recall the following algebraic constructions:

\begin{definition}
  Let $R$ be an algebra over $k$, and let $S$ be a set. The free $R$-bimodule $R S R$ with basis $S$ is just the free left $R^{\mathrm{e}}$-module over $S$, where $R^{\mathrm{e}} = R \otimes_k R^{\mathrm{op}}$. The free $R$-ring $R \langle S \rangle$ generated by $S$ is the tensor algebra
  \begin{equation*}
    R \langle S \rangle = R \oplus M \oplus (M \otimes_R M) \oplus (M \otimes_R M \otimes_R M) \oplus \dotsb 
  \end{equation*}
  of the $R$-bimodule $M = R S R$.
\end{definition}

If we write $(a \otimes_k b)s \in R S R$ for $a, b \in R$ and $X \in S$ as $a X b$, then every element of $R \langle S \rangle$ is a finite sum of elements of the form $a_1 X_1 a_2 \dotsb a_n X_n a_{n+1}$ for some $a_i \in R$ and $X_i \in S$. We note that an element of $R$ and an element of $S$ do not commute in $R \langle S \rangle$, but an element of $k$ $(\subset R)$ is central in $R \langle S \rangle$. Thus $R \langle S \rangle$ is an algebra over $k$. One can define ``the $R$-ring generated by $S$ subject to the relations ...'' as the quotient algebra of $R \langle S \rangle$ by the ideal defined by the relations.

Now let $(X, \sigma)$ be a dynamical Yang-Baxter map satisfying \eqref{eq:Shibukawa-condition}. To avoid technical difficulty, we only consider the case where $\Lambda$ is finite. Then $\mathcal{F}$ is a separable algebra with Frobenius system
\begin{equation}
  \label{eq:F-Frob-system}
  t(f) = \sum_{\lambda \in \Lambda} f(\lambda)
  \quad (f \in \mathcal{F})
  \quad \text{and} \quad
  e = \sum_{\lambda \in \Lambda} \delta_{\lambda} \otimes \delta_{\lambda}
  \in \mathcal{F} \otimes_k \mathcal{F},
\end{equation}
where $\delta_{\lambda} \in \mathcal{F}$ is the function defined by $\delta_{\lambda}(\mu) = \delta_{\lambda\mu}$ for $\mu \in \Lambda$. By the result of Schauenburg \cite{MR2024429}, we may regard a bialgebroid over $\mathcal{F}$ as a weak bialgebra over $k$. We also remark that Shibukawa has considered the equation \eqref{eq:set-th-DYBE-2} in \cite{MR3448180} instead of the set-theoretical dynamical Yang-Baxter equation~\eqref{eq:set-th-DYBE} considered in this paper. For this reason, we use the notation~\eqref{eq:set-th-DYBE-notation} and define
\begin{equation*}
  \check{\sigma}: \Lambda \times X \times X \to X \times X,
  \quad (\lambda, x, y) \mapsto (y \dybmr{\lambda} x, y \dybml{\lambda} x)
\end{equation*}
so that $\check{\sigma}$ satisfies \eqref{eq:set-th-DYBE-2-inv-cond} and \eqref{eq:set-th-DYBE-2} (see Remark~\ref{rem:set-th-DYBE-2}). The weak bialgebra $\mathfrak{B}(\sigma)$ introduced in the below is in fact Shibukawa's bialgebroid $A_{\check{\sigma}}$ over $\mathcal{F}$ regarded as a weak bialgebra over $k$ in Schauenburg's manner.

\begin{definition}
  We set $\mathcal{E} := \mathcal{F} \otimes_k \mathcal{F}$ and write $\xi \otimes 1 \in \mathcal{E}$ and $1 \otimes \xi \in \mathcal{E}$ simply by $\xi$ and $\overline{\xi}$, respectively. The weak bialgebra $\mathfrak{B}(\sigma)$ is the $\mathcal{E}$-ring generated by the formal symbols $L^{+}_{a b}$, $L_{a b}^{-}$ ($a, b \in X$) subject to the relations
  \begin{gather}
    \label{eq:Shibukawa-rel-1}
    \sum_{b \in X} L_{a b}^{+} L_{b c}^{-} = 1 = \sum_{b \in X} L_{a b}^{-} L_{b c}^{+}, \\
    \label{eq:Shibukawa-rel-2}
    L_{a b}^{+} \cdot \xi = (a \triactl \xi) \cdot L_{a b}^{+}, \quad
    L_{a b}^{+} \cdot \overline{\xi} = \overline{b \triactl \xi} \cdot L_{a b}^{+}, \\
    \label{eq:Shibukawa-rel-3}
    L_{a b}^{-} \cdot (b \triactl \xi) = \xi \cdot L_{a b}^{-}, \quad
    L_{a b}^{-} \cdot \overline{a \triactl \xi} = \overline{\xi} \cdot L_{a b}^{-}, \\
    \label{eq:Shibukawa-rel-4}
    \sum_{x, y \in X} \sigma^{x y}_{a b} \, L_{x c}^{+} L_{y d}^{+}
    = \sum_{x, y \in X} \overline{\sigma^{c d}_{x y}} \, L_{a x}^{+} L_{b y}^{+}
  \end{gather}
  for $a, b, c, d \in X$ and $\xi \in \mathcal{F}$, where
  \begin{equation*}
    (a \triactl \xi)(\lambda) = \xi(\lambda \dyact a)
    \quad (a \in X, \xi \in \mathcal{F}, \lambda \in \Lambda)
  \end{equation*}
  and $\sigma^{p q}_{r s} \in \mathcal{F}$ for $p, q, r, s \in X$ is defined by
  \begin{equation}
    \label{eq:Shibukawa-sigma-def}
    \sigma^{p q}_{r s}(\lambda)
    =
    \begin{cases}
      1 & \text{if $\sigma(\lambda, p, q) = (r, s)$}, \\
      0 & \text{otherwise}.
    \end{cases}
  \end{equation}
  The comultiplication $\Delta: \mathfrak{B}(\sigma) \to \mathfrak{B}(\sigma) \otimes_k \mathfrak{B}(\sigma)$ is determined by
  \begin{gather*}
    \Delta(\xi) = \sum_{\lambda \in \Lambda} \xi \overline{\delta}_{\lambda} \otimes \delta_{\lambda},
    \quad
    \Delta(\overline{\xi}) = \sum_{\lambda \in \Lambda} \overline{\delta}_{\lambda} \otimes \overline{\xi} \delta_{\lambda}, \\
    \Delta(L_{a b}^{\pm}) = \sum_{c \in X, \lambda \in \Lambda} \overline{\delta}_{\lambda} L_{a c}^{\pm} \otimes \delta_{\lambda} L_{c b}^{\pm}
  \end{gather*}
  for $\xi \in \mathcal{F}$ and $a, b \in X$. To define the counit, we introduce the left action $\rightharpoonup$ of the algebra $\mathfrak{B}(\sigma)$ on the vector space $\mathcal{F}$ defined by
  \begin{equation*}
    \xi \rightharpoonup f = \xi f = \overline{\xi} \rightharpoonup f
    \quad \text{and} \quad
    (L_{a b}^{\pm} \rightharpoonup f)(\lambda)
    = \delta_{a b} f(\lambda \dyact a^{\pm 1})
  \end{equation*}
  for $\xi \in \mathcal{F}$ and $a, b \in X$, where $(-) \dyact a^{+1} = (-) \dyact a$ and $(-) \dyact a^{-1}$ is the inverse of the map $(-) \dyact a$. Now the counit $\varepsilon: \mathfrak{B}(\sigma) \to k$ is defined by
  \begin{equation*}
    \varepsilon(b) = t(b \rightharpoonup 1_{\mathcal{F}}) \quad (b \in \mathfrak{B}(\sigma)),
  \end{equation*}
  where $t: \mathcal{F} \to k$ is the Frobenius trace of $\mathcal{F}$ given by~\eqref{eq:F-Frob-system} and $1_{\mathcal{F}} = \sum_{\lambda \in \Lambda} \delta_{\lambda}$ is the unit element of $\mathcal{F}$.
\end{definition}

\subsection{Comparison of two constructions}

Let $\Lambda$ be a non-empty finite set, and let $(X, \sigma)$ be a dynamical Yang-Baxter map satisfying~\eqref{eq:Shibukawa-condition}. On the one hand, we have the weak bialgebra $\mathfrak{B}(\sigma)$ by Shibukawa's construction. On the other hand, we have a weak bialgebra by composing Theorem~\ref{thm:BrQ} and Hayashi's construction: Namely, by the theorem, we have a braided object $(Q, \widetilde{\sigma}) := \Br(\mathsf{Q})(X, \sigma)$ of $\Quiv{\Lambda}$. Set $W = \mathrm{span}_k(Q)$. The map $\widetilde{\sigma}$ induces a linear map
\begin{equation*}
  w_{\sigma}: W \otimes_{\mathcal{F}} W \to W \otimes_{\mathcal{F}} W,
  \quad w_{\sigma}(a \otimes b) = c \otimes d
  \quad ((a, b) \in Q^2, (c, d) = \widetilde{\sigma}(a, b)),
\end{equation*}
which gives rise to a star-triangular face model $(Q, w_{\sigma})$. We thus obtain a weak bialgebra $\mathfrak{A}(w_{\sigma})$ from the braided object $(X, \sigma)$. The main result of this section is the following relation between these two weak bialgebras:

\begin{theorem}
  \label{thm:two-weak-bialg}
  For an integer $m \ge 0$ and two elements
  \begin{equation}
    \label{eq:thm-comparison-paths}
    p = ((\lambda_1, x_1), \dotsc, (\lambda_m, x_m))
    \quad \text{and} \quad
    q = ((\mu_1, y_1), \dotsc, (\mu_m, y_m))
  \end{equation}
  of $Q^{(m)}$, we set
  \begin{equation}
    \label{eq:w-bialg-map-def}
    \phi \left( \E{p}{q} \right) = \delta_{\lambda_1} \overline{\delta}_{\mu_1} L_{x_1 y_1}^+ \dotsb L_{x_m y_m}^+.
  \end{equation}
  Then $\phi$ extends to a morphism $\phi: \mathfrak{A}(w_{\sigma}) \to \mathfrak{B}(\sigma)$ of weak bialgebras.
\end{theorem}
\begin{proof}
  We shall check that $\phi$ preserves the relations~\eqref{eq:Hayashi-rel-0}--\eqref{eq:Hayashi-rel-2}. It is easy to see that $\phi$ preserves \eqref{eq:Hayashi-rel-0}. By \eqref{eq:Shibukawa-condition}, \eqref{eq:Shibukawa-rel-2} and~\eqref{eq:Shibukawa-rel-3}, we have
  \begin{equation*}
    \delta_{\lambda} L_{x y}^{+}
    = (x \triactl \delta_{\lambda \dyact x}) L_{x y}^{+}
    = L_{x y}^{+} \delta_{\lambda \dyact x}
    \text{\quad and \quad}
    \overline{\delta}_{\mu} L_{x y}^{+} = L_{x y}^{+} \overline{\delta}_{\mu \dyact x}
  \end{equation*}
  for all $x, y \in X$ and $\lambda, \mu \in \Lambda$. By using these equations and the definition of the quiver $Q = \mathsf{Q}(X)$, one can verify that $\phi$ preserves \eqref{eq:Hayashi-rel-1}.

  We now check that $\phi$ preserves \eqref{eq:Hayashi-rel-2}. For this purpose, we determine the symbol $\mathbf{w}[ \, \cdots \, ]$ associated to the star-triangular face model $(Q, w_{\sigma})$ arising from $(X, \sigma)$. Suppose that $((\lambda, a), (\lambda', b))$ and $((\mu, c), (\mu', d))$ are elements of $Q^{(2)}$. By using $\sigma^{a b}_{c d} \in \mathcal{F}$ defined by~\eqref{eq:Shibukawa-sigma-def}, the symbol $\mathbf{w}[ \, \cdots \, ]$ is expressed as follows:
  \begin{equation*}
    \W{(\lambda,a)}{(\lambda',b)}{(\mu,c)}{(\mu',d)} = \sigma^{a b}_{c d}(\lambda) \delta_{\lambda, \mu}.
  \end{equation*}
  Thus we compute:
  \begin{align*}
    & \sum_{((\lambda,x),(\lambda',y)) \in Q^2}^{}
      \W{(\lambda,x)}{(\lambda',y)}{(\mu,a)}{(\mu',b)}
      \phi \left( \E{(\lambda,x)}{(\nu,c)} \right)
      \phi \left( \E{(\lambda',y)}{(\nu',d)} \right) \\
    & = \sum_{\lambda \in \Lambda, x, y \in X}^{}
      \sigma_{a b}^{x y}(\lambda) \delta_{\mu}(\lambda) \delta_{\lambda}
      \overline{\delta}_{\gamma} L_{x c}^+ L_{y d}^+
      = \sum_{x, y \in X}
      \sigma_{a b}^{x y} \, \delta_{\mu} \overline{\delta}_{\gamma} L_{x c}^+ L_{y d}^+ \\
    & = \delta_{\mu} \overline{\delta}_{\gamma}
      \cdot (\text{the left-hand side of \eqref{eq:Shibukawa-rel-4}}), \\
    & \sum_{((\lambda,x),(\lambda',y)) \in Q^2}^{}
      \W{(\nu,c)}{(\nu',d)}{(\lambda,x)}{(\lambda',y)}
      \phi \left( \E{(\mu,a)}{(\lambda,x)} \right)
      \phi \left( \E{(\mu',b)}{(\lambda',y)} \right) \\
    & = \sum_{\lambda \in \Lambda, x, y \in X}^{}
      \sigma_{x y}^{c d}(\lambda) \delta_{\nu}(\lambda) \delta_{\mu} \overline{\delta}_{\lambda}
      L_{a x}^+ L_{b y}^+
      = \sum_{\lambda \in \Lambda, x, y \in X}^{}
      \overline{\sigma_{x y}^{c d}} \, \delta_{\mu} \overline{\delta}_{\nu}
      L_{a x}^+ L_{b y}^+ \\
    & = \delta_{\mu} \overline{\delta}_{\gamma}
      \cdot (\text{the right-hand side of \eqref{eq:Shibukawa-rel-4}}).
  \end{align*}
  Namely, $\phi$ preserves \eqref{eq:Hayashi-rel-2}. Hence $\phi$ extends to a morphism $\phi: \mathfrak{A}(w_{\sigma}) \to \mathfrak{B}(\sigma)$ of algebras over $k$.

  To complete the proof, we check that the algebra map $\phi$ preserves the coalgebra structure of $\mathfrak{A}(w_{\sigma})$. Since the comultiplication of a weak bialgebra is multiplicative, it is sufficient to show
  \begin{equation*}
    \Delta \phi \left( \E{(\lambda,x)}{(\mu,y)} \right)
    = (\phi \otimes \phi) \Delta \left( \E{(\lambda,x)}{(\mu,y)} \right)
  \end{equation*}
  for all $\lambda, \mu \in \Lambda$ and $x, y \in X$ to show that $\phi$ preserves the comultiplication. This is verified as follows:
  \begin{align*}
    \Delta \phi \left( \E{(\lambda,x)}{(\mu,y)} \right)
    & = \sum_{\substack{\nu_1, \nu_2, \nu_3 \in \Lambda, \\ c \in X}}
      (\delta_{\lambda} \overline{\delta}_{\nu_1} \otimes \delta_{\nu_1})
      \cdot (\overline{\delta}_{\nu_2} \otimes \overline{\delta}_{\mu} \delta_{\nu_2})
      \cdot (\overline{\delta}_{\nu_3} L_{x c}^{+} \otimes \delta_{\nu_3} L_{c y}^{+}) \\
    & = \sum_{\nu \in \Lambda, c \in X} \delta_{\lambda} \overline{\delta}_{\nu} L_{x c}^{+}
      \otimes \overline{\delta}_{\mu} \delta_{\nu} L_{c y}^{+} \\
    & = \sum_{\nu \in \Lambda, c \in X} \phi \left( \E{(\lambda,x)}{(\mu,c)} \right)
      \otimes \phi \left( \E{(\nu, c)}{(\mu,y)} \right) \\
    & = (\phi \otimes \phi) \Delta \left( \E{(\lambda,x)}{(\mu,y)} \right).
  \end{align*}
  Finally, we check that $\phi$ preserves the counit. Let $p$ and $q$ be paths on $Q$ of length $m$ as in \eqref{eq:thm-comparison-paths}. Then, by the definition of the quiver $Q$, we have
  \begin{equation}
    \label{eq:thm-comparison-paths-cond}
    \lambda_{i + 1} = \lambda_i \dyact x_i
    \quad \text{and} \quad
    \mu_{i + 1} = \mu_i \dyact y_i
    \quad (i = 1, \dotsc, m -1).
  \end{equation}
  Since $a \triactl 1_{\mathcal{F}} = 1_{\mathcal{F}}$ for all $a \in X$, we have
  \begin{align*}
    \phi \left( \E{p}{q} \right) \rightharpoonup 1_{\mathcal{F}}
    & = \Big( \delta_{\lambda_1} \overline{\delta}_{\mu_1} L_{x_1 y_1}^+ \dotsb L_{x_m y_m}^+ \Big) \rightharpoonup 1_{\mathcal{F}} \\
    & = \delta_{x_1,y_1} \delta_{x_2,y_2} \dotsb \delta_{x_m,y_m} \delta_{\lambda_1} \delta_{\mu_1},
  \end{align*}
  where $\rightharpoonup$ is the action of $\mathfrak{B}(\sigma)$ on the vector space $\mathcal{F}$ used in the definition of the counit of $\mathfrak{B}(\sigma)$. Hence,
  \begin{equation*}
    \varepsilon \phi \left( \E{p}{q} \right)
    = \delta_{x_1,y_1} \delta_{x_2,y_2} \dotsb \delta_{x_m,y_m} \delta_{\lambda_1,\mu_1}.
  \end{equation*}
  In view of \eqref{eq:thm-comparison-paths-cond}, the right-hand side of this equation is $1$ if $p = q$ and $0$ otherwise. Thus we conclude that $\varepsilon \phi(a) = \varepsilon(a)$ for all $a \in \mathfrak{A}(w_{\sigma})$.
\end{proof}

\begin{remark}
  \label{rem:two-weak-bialg}
  (1) Let $\mathfrak{B}^{+}(\sigma)$ be the $\mathcal{E}$-ring generated by $L_{a b}^{+}$ ($a, b \in X$) subject to the relations~\eqref{eq:Shibukawa-rel-2} and~\eqref{eq:Shibukawa-rel-4}. The algebra $\mathfrak{B}^{+}(\sigma)$ is a weak bialgebra by the coalgebra structure defined in the same way as $\mathfrak{B}(\sigma)$. The weak bialgebra $\mathfrak{A}(w_{\sigma})$ is isomorphic to $\mathfrak{B}^{+}(\sigma)$. Indeed, the algebra map $\mathfrak{A}(w_{\sigma}) \to \mathfrak{B}{}^{+}(\sigma)$ defined by the same formula as~\eqref{eq:w-bialg-map-def} has the inverse determined by
  \begin{equation*}
    L_{a b}^{+} \mapsto \sum_{\lambda, \mu \in \Lambda} \E{(\lambda, a)}{(\mu, b)},
    \quad
    \xi_1^{} \overline{\xi}_2 \mapsto \sum_{\lambda, \mu \in Q^{(0)}} \xi_1^{}(\lambda) \xi_2^{}(\mu) \E{\lambda}{\mu}
  \end{equation*}
  for $a, b \in X$ and $\xi_1, \xi_2 \in \mathcal{F}$. The algebra $\mathfrak{B}(\sigma)$ has extra generators $L_{a b}^{-}$ and hence is not isomorphic to $\mathfrak{A}(w_{\sigma})$ in general.

  (2) Let $(W, w)$ be a star-triangular face model. Hayashi \cite{MR1623965} showed that the weak bialgebra $\mathfrak{A}(w)$ has a Hopf closure $\mathfrak{H}(w)$ under the assumption that $w$ is {\em closable} \cite[Section 3]{MR1623965}. We recall that, as an algebra, $\mathfrak{H}(w)$ is generated by
  \begin{equation*}
    \E{p}{q}, \quad \E{\overline{p}}{\overline{q}} \quad (p, q \in Q^{(2)})
  \end{equation*}
  subject to the relations like \eqref{eq:Hayashi-rel-0}--\eqref{eq:Hayashi-rel-2}. Notably, Hayashi introduced an extension $\mathbf{w}_{\mathrm{LD}}$ of $\mathbf{w}$, called the Lyubashenko double, and showed that the equation
  \begin{equation}
    \label{eq:Hayashi-rel-3}
    \newcommand{\olE}[2]{\mathop{%
        \renewcommand{\arraystretch}{1.2}%
        \setlength{\arraycolsep}{1pt}%
        \mathbf{e} \! \left[
          \begin{array}{c} {\mathstrut \overline{#1}} \\ {\mathstrut \overline{#2}} \end{array}
        \right]}}
    \newcommand{\olWLD}[4]{\facewt{\mathbf{w}_{\mathrm{LD}}}%
      {\overline{#1}}{\overline{#2}}{\overline{#3}}{\overline{#4}}}
    \sum_{(x, y) \in Q^{(2)}} \olWLD{x}{y}{a}{b} \olE{x}{c} \olE{y}{d}
    = \sum_{(x, y) \in Q^{(2)}} \olWLD{c}{d}{x}{y} \olE{a}{x} \olE{b}{y}
  \end{equation}
  and its variants hold in $\mathfrak{H}(w)$ \cite[Proposition 5.5]{MR1623965}. Now we consider the case where $(Q, w_{\sigma})$ is a star-triangular face model arising from a dynamical Yang-Baxter map $(X, \sigma)$ satisfying~\eqref{eq:Shibukawa-condition}. Shibukawa observed that $\mathfrak{B}(\sigma)$ has an antipode (and thus it is a weak Hopf algebra by \cite{MR2024429}) when $\sigma$ is {\em rigid} \cite[Section 3]{MR3448180}. Since we do not know any relations between Hayashi's closability and Shibukawa's rigidity, we assume that, for the time being, $w_{\sigma}$ is closable and $\sigma$ is rigid. Then, by the universal property of the Hopf closure, one can extend the map $\phi: \mathfrak{A}(w_{\sigma}) \to \mathfrak{B}(\sigma)$ of the above theorem to a map $\mathfrak{H}(w_{\sigma}) \to \mathfrak{B}(\sigma)$ of weak Hopf algebras. This map does not seem to be an isomorphism in general. The difficulty is that $\mathfrak{B}(\sigma)$ has no defining relations corresponding to \eqref{eq:Hayashi-rel-3} and its variants.
\end{remark}


\end{document}